\documentclass[12pt]{amsart}
\usepackage{graphicx}
\usepackage{amssymb}
\usepackage{amsthm}
\usepackage{listings}
\usepackage{lineno}
\usepackage[margin=3cm]{geometry}
\usepackage[all,cmtip, color,matrix,arrow]{xy}
\usepackage{amsaddr}
\usepackage{tikz-cd}
\usepackage{amsmath}%To use \text 
\usepackage[utf8]{inputenc}
\usepackage[pdfpagelabels]{hyperref}
\usepackage[capitalize]{cleveref}
\crefname{thm}{Theorem}{Theorems}
\usepackage[export]{adjustbox}
\usepackage{todonotes}
\usepackage{bm}
\usepackage{wrapfig}
\usepackage{float}
\usepackage{mathtools}
\usepackage{aliascnt}
\newaliascnt{eqfloat}{equation}
\newfloat{eqfloat}{h}{eqflts}
\floatname{eqfloat}{Equation}
\usepackage[mathscr]{euscript}

\newcommand*{\ORGeqfloat}{}
\let\ORGeqfloat\eqfloat
\def\eqfloat{%
  \let\ORIGINALcaption\caption
  \def\caption{%
    \addtocounter{equation}{-1}%
    \ORIGINALcaption
  }%
  \ORGeqfloat
}

\makeatletter
\providecommand*{\shuffle}{%
  \mathbin{\mathpalette\shuffle@{}}%
}
\newcommand*{\shuffle@}[2]{%
  \sbox0{$#1\vcenter{}$}%
  \kern .15\ht0 % side bearing
  \rlap{\vrule height .25\ht0 depth 0pt width 2.5\ht0}%
  \raise.1\ht0\hbox to 2.5\ht0{%
    \vrule height 1.75\ht0 depth -.1\ht0 width .17\ht0 %
    \hfill
    \vrule height 1.75\ht0 depth -.1\ht0 width .17\ht0 %
    \hfill
    \vrule height 1.75\ht0 depth -.1\ht0 width .17\ht0 %
  }%
  \kern .15\ht0 % side bearing
}
\makeatother

\newcommand{\qshuffle}{\,\underline{\shuffle}\,}

\theoremstyle{definition}
\newtheorem{thm}{Theorem}[section]
\newtheorem{prop}[thm]{Proposition}
\newtheorem{lm}[thm]{Lemma}

\newtheorem{defin}[thm]{Definition}
\newtheorem{smpl}[thm]{Example}

\newtheorem{conj}[thm]{Conjecture}
\newtheorem{rem}[thm]{Remark}
\crefname{lm}{Lemma}{Lemmas}
\crefname{thm}{Theorem}{Theorems}
\crefname{prop}{Proposition}{Propositions}
\crefname{defin}{Definition}{Definitions}
\crefname{rem}{Remark}{Remarks}

\DeclareMathOperator{\im}{im}
\DeclareMathOperator{\id}{id}

\DeclareMathOperator{\rk}{\mathrm{rk}}

\DeclareMathOperator{\spn}{\mathrm{span}}

\DeclareMathOperator{\Dsign}{\mathscr{S}}

\newcommand{\R}{\mathbb{R}}

\newcommand{\polyd}{\mathfrak A_d}
\newcommand{\polydh}{\mathfrak A_{d}^{\leq h}}

\newcommand{\KK}{\mathtt{K}}

\newcommand{\lb}{[\![}
\newcommand{\rb}{]\!]}
\newcommand{\V}{\mathcal V}
\newcommand{\U}{\mathcal U}

\newcommand{\bv}{\mathsf{b}}
\newcommand{\vv}{\mathsf{v}}
\newcommand{\vw}{\mathsf{w}}
\newcommand{\lv}{\mathfrak{l}}

\newcommand{\el}{\mathsf{e}}

\newcommand{\MS}{\mathrm{MS}}

\newcommand{\ot}{\otimes}

\newcommand{\LL}{\mathtt{L}}
\newcommand{\uw}[1]{#1}

\begin{document}

\title{Discrete signature varieties} % Subtitle
\author{Carlo Bellingeri}
\address{IECL, Université de Lorraine, France}
\email{carlo.bellingeri@univ-lorraine.fr}

\author{Raul Penaguiao}
\address{HAEUSLER AG, Switzerland}
\email{rpenaguiao@haeusler.ch}

\date{} % Date

\begin{abstract}
Discrete signatures are invariants computed from time series corresponding to the discretised version of the signature of paths. We study the algebraic varieties arising from their images, the discrete signature varieties. We introduce them and compute their dimension in many cases. From a particular subclass of these varieties, we derive a partial solution to the Chen-Chow theorem for complex-valued time series.

\end{abstract}

\maketitle
\tableofcontents
\section{Introduction}

Signatures of smooth curves were introduced in \cite{chen1957integration}, leading to revolutionary applications on stochastic analysis with contributions by Lyons, Friz, Hairer, and others, see \cite{lyons2007differential,frizbook,friz2020course}.

The discrete counterpart of these quantities is the \textbf{discrete signature} of a time series, or iterated sum signature, introduced by \cite{Tapia20}. From an informal point of view, these quantities capture one essential feature of a time series that is time-warping invariance. That is, any  addition of consecutive repeated terms in the time series does not change the resulting value.
%When focusing on data that is naturally discrete, several interesting applications of this new invariant arise, of which we remark signal compression %\cite{bandeira2017estimation}. 
More concretely, for any given time series $x\colon \{1, \ldots, N\}\to \KK^d $, $x = (x_1, \ldots, x_N)$ for some field $\KK$ of characteristic $0$, the discrete signature (denoted by $\Dsign(x)$)  is a special element in the space of tensor series, spanned by infinite sums over tensors of monomials. For any family  $p_1\,,  \ldots \,, p_l$ of non-constant monomials $p_i$ in $\KK[X_1, \ldots , X_d]$   the corresponding coefficient  of $\Dsign(x)$ is defined as follows:
\begin{equation*}\label{eq:dss}
\langle\Dsign(x) , p_1\otimes \cdots \otimes p_l \rangle = \sum_{1 \leq i_1 < \dots < i_l < N} p_1(x_{i_1+1} - x_{i_1}) \cdots p_l(x_{i_l+1} - x_{i_l}) \, .
\end{equation*}
Using the notation $x_j= (x^1_j\,, \ldots \,, x_j^d)$ for each value of $x$ and $[i_1\ldots i_k] $ for the monomial $X_{i_1}\ldots X_{i_k}$ we have the identities:
\begin{align*}
\langle\Dsign(x), [i]\rangle =&\sum_{k=1}^{N-1} (x^i_{k+1} - x^i_{k})= x_N^i- x_1^i\,,\\\quad \langle\Dsign(x),[ij]\rangle =&\sum_{k=1}^{N-1}  (x^i_{k+1} - x^i_{k})(x^j_{k+1} - x^j_{k})\,,\\ \langle\Dsign(x),[i]\otimes [j] \rangle =& \sum_{k = 1}^{N-1}\sum_{l = k+1}^{N-1} (x_{k+1}^i - x_{k}^i)(x_{l+1}^j - x_{l}^j)\,.
\end{align*} 
We argue that the set of tensor series formed by discrete signatures is interesting when studied from the point of view of algebraic geometry. To state it precisely, we fix an integer $h\geq 1$ and consider the projection of $\Dsign^{h}(x)$ on the space of tensors polynomials with weight equal $h$, see \cref{sec:prelim} for the notion of weight and projection. The discrete signature can be seen as a polynomial map $x\mapsto \Dsign^{h}(x)$ on the $Nd$ parameters of the time series with values in a finite-dimensional vector space. Our goal is to describe the image of $\Dsign^{h}$ after this projection. Since the image of a polynomial map is not always an algebraic set, we consider the Zariski closure of this set, and we can define the \textbf{variety of discrete signatures $\V_{d, h, N}$}. This variety is an irreducible algebraic variety that arises from an implicitization problem. That is, this variety is represented by parametric equations, and we are interested in recovering the implicit equations that determine it. This is a classical problem in algebraic geometry which is in general very difficult both from a theoretical and computational point of view, see e.g. \cite[Chapter 4]{michalek2021invitation}. The main goal of this paper is to introduce this family of varieties. In particular, we also want to describe the limit of these varieties when $N$ is large, which gives us the \textbf{universal variety $\V_{d, h}$}.

 In many aspects, our results can be related to \cite{amendola2019varieties}, but two main differences arise between the study of signatures varieties and the discrete signatures varieties.

  A first difference is that classical signatures satisfy \textbf{shuffle relations}, see \cite{chen1957integration}, whereas the latter satisfy \textbf{quasi-shuffle relations}. See \cite{hoffman2000} for an in-depth assessment of quasi-shuffles. This difference is encoded in the groups $\hat{\mathcal{G}}^{\leq h}(\polyd)$, which were studied in \cite{Bellingeri2022} in the contexts of smooth quasi-geometric rough paths.  In this paper, we explore properties of the level $h$ family of varieties $\hat{\mathcal{G}}_{d,h}$, see \cref{def:quasi-shuffle-variety}, which we can naively think as a special submanifold of $\hat{\mathcal{G}}^{\leq h}(\polyd)$ containing the algebraic varieties of interest $\mathcal{V}_{d,h}$.

A second difference between Chen's signatures and discrete signatures is related to the absence, at the moment, of a \textbf{Chen-Chow theorem} for discrete signatures, see \cite[Section 3.2]{diehl_tapia_EF_2020}. This theorem is central in \cite{amendola2019varieties} to describe the universal varieties in classical signature varieties. To better analyse the universal varieties in our discrete signatures setting,  we were able to solve the Chen-Chow theorem in the particular case $h=2$. This task was achieved by introducing the notion of \textbf{reachability}, a notion which allows us to have an equivalent formulation of the Chen-Chow theorem by solving a system of polynomial equations. Fundamental for the combinatorics of this formulation are \textbf{Lyndon words}. Specifically, we use a particular grading of the original Lyndon words from \cite{chen1958free}, yielding a different factorization theorem.  These Lyndon words with this grading arise in several places of the Hopf algebra landscape, see e.g.\cite{borga2020feasible, penaguiao2022pattern,penaguiao2020algebraic}.

Our main contributions, stated below in  \cref{thm:dimension,lm:h2constructionX}, establish the following result: First we compute the dimension of $\hat{\mathcal{G}}_{d,h}$.
\begin{thm}\label{thm:main}
Fix $d, h$ integers $\geq 1$. The dimension of the variety $\hat{\mathcal{G}}_{d,h}$ is given by 
\begin{equation}\label{eq:dimension}
\left(\sum_{l=1}^h\sum_{k|l} \frac{\mu\left(k\right)}{k}  \sum_{\alpha\in C(l/k)} \frac{1}{\ell (\alpha)} \prod_{i=1}^{\ell(\alpha)} \binom{\alpha_i + d - 1}{d - 1}\right) - 1 \,,
\end{equation}
where $\mu $ is the M\"obius function, $C(k)$ is the set of compositions of the integer $k$ and $\ell(\alpha)$ is the length of each composition.
\end{thm}
In addition to this result, we also link $\mathcal{V}_{d,h}$ with $\hat{\mathcal{G}}_{d,h}$ in a special case.
\begin{thm}\label{thm:main2}
We have the identity $\hat{\mathcal{G}}_{d,h}=\mathcal{V}_{d,h}$ in the case $d\geq 1$ and $h=2$.
\end{thm}
We conjecture that the equality $\hat{\mathcal{G}}_{d,h}=\mathcal{V}_{d,h}$ should hold for any choice of the parameters $d,h$ and a fortiori the formula \eqref{eq:dimension} should tell us the dimension of $\mathcal V_{d, h}$.

Furthermore, we also present in 
\cref{sec:computations} some computational results about the varieties $\mathcal V_{d, h,N}$  when the parameters $d,h,N$ are low.

We now state the organisation of the paper. 
Preliminaries and main definitions will be presented in \cref{sec:prelim}. After introducing the main properties of quasi-shuffle algebras in \cref{sec:computations}, we define the varieties of discrete signatures.  For some specific values $h$ and $d$, we leverage computational algebra tools to present a \textbf{ Gr\"obner basis} of the vanishing ideals of $\V_{d, h, N}$, drawing inspiration from algebraic statistics, see \cite{drton2008lectures}. 
From the explicit knowledge of a Gr\"obner basis, we compute numerically first invariants of $\V_{d, h, N}$ like dimension and degree.  Then, in \cref{sec_Lie}, we present an enumerative result for the number of Lyndon words with some fixed weight, from which we can deduce Theorem \ref{thm:main}. Finally, in \cref{sec:conj}, we prove a characterisation of the universal variety under the notion of reachability, and prove Theorem \ref{thm:main2}.

We hope in future works to describe new properties of these varieties, like a general computation of their degree and, more generally, possible extensions of the theory in new contexts, such as the study of varieties defined from time series defined over non-zero characteristic fields or even stochastic. A non-exhaustive list of possible applications and future works was added at the end of the article.

\section{Preliminaries}
\label{sec:prelim}

\subsection*{Combinatorics}
We will adopt the notation $\mathbb{N}= \{1\,, 2\,, \ldots\}$. Given a countable alphabet $I$, a \textbf{word} $w = (i_1, \dots , i_n)$ is an $n$-tuple of elements in $I$ with $n\in \mathbb{N}$. We write $w = i_1  \ldots  i_n $ for simplicity and $|w | = n$ is  called the length of $w$. We denote by  $\mathcal W (I)$ the set of words in $I$, including the empty word, which we denote by $\varepsilon$. Two words $w $ and $v$ may be concatenated, which we represent by $w v$. We call this operation \textbf{concatenation}. We always equip an alphabet $I$ with a \textbf{grading} map $\text{grad} \colon I \to \mathbb{N}$. For any word $w = i_1  \ldots   i_n$ we call the sum $\sum_{j=1}^n \text{grad}(i_j)$ the \textbf{weight}  of a  $w$ and we  denote it by $||w||$. One has that $|w|\leq ||w||$ for any non-empty word $w$ and by definition we set $|\varepsilon|=||\varepsilon||=0$.

Given a countable set $A$, a \textbf{multiset} $S \subset A$ is a collection of elements in $A$, allowing for repetitions. We denote the family of multisets of elements in $A$ by $\MS_A^0$. This includes the empty multiset $\{\emptyset \}$. We denote $\MS_A \coloneqq \MS_A^0 \setminus \{\emptyset \}$.  When $A = \{1\,, \ldots\, d\}$ for some $d\in \mathbb{N}$, we denote the sets $\MS_A^0, \MS_A$ by $\MS_d^0$, $\MS_d$, respectively. We will denote all the elements of $\MS_d$ into square brackets, e.g. $[1122334]\in \MS_4$. The alphabet $\MS_d$ has an associative and commutative operation over it, the union and a natural grading map given by the cardinality of the multiset. Note that in the context of multisets, the union counts repetitions of elements. Moreover, $\MS_d$ admits many total orders preserving the usual ordering of $\{1\,, \ldots\, d\}$. This can be easily constructed by taking a bijection between multisets of $\{1\,, \ldots\, d\}$ and monomials on $d$ variables and then applying a total monomial order such as the \textit{graded lexicographic order}, see e.g. \cite{miller2005combinatorial}.

A \textbf{composition} $\alpha$ of an integer $k$ is an $n$-tuple of positive integers $\alpha = (\alpha_1, \dots , \alpha_n)$ such that $\sum_{i=1}^n \alpha_i = k$. We denote the set of all compositions of $k$ by $C(k)$ and write $\ell(\alpha) = n$ for the \textbf{length} of the composition. We also use the following quantities on compositions $\alpha! \coloneqq \prod_{i=1}^n \alpha_i !$ and $\Pi \alpha \coloneqq \prod_{i=1}^n \alpha_i$. 

Assume now that $I=\MS_d$. If $\alpha \in C(k) $  and $w = i_1 \ldots  i_k\in \mathcal{W}(\MS_d)$ is a word such that $|w| = k$, then we define the \textbf{contracted word} $(w)_{\alpha}$ by applying the intrinsic operation of $\MS_d$ to the letters in $w$ according to $\alpha$, i.e.
\[(w)_{\alpha} = \tau_1 \ldots \tau_{\ell(\alpha)}\,,  \quad \text{with} \quad \tau_j \coloneqq [i_{s_j + 1} \cdots i_{s_j+ \alpha_j}]\, ,\] 
where $s_j = \sum_{i=1}^{j-1} \alpha_i$ for $j = 1, \dots , \ell(\alpha)$. For instance, using the same notations in the introduction, if we take $w = [1244][2][13][22]$ and $\alpha = (2, 2)$ we get $(w)_{\alpha} = [12244] [1322]$.

\subsection*{Tensor algebras}

Let $V$ be a vector space over a field $\KK$ of characteristic zero. We define  the \textbf{tensor algebra} $T(V)$ and the \textbf{tensor series} $T((V))$ as follows:
\begin{equation}\label{tensor}
T(V) \coloneqq \bigoplus_{k=0}^{\infty}V^{\otimes k} \quad \quad \quad T((V)) \coloneqq \prod_{k=0}^{\infty}V^{\otimes k}  \,,
\end{equation}
with $ V^{\otimes 0}:=\KK $.
By fixing a basis $\mathcal{B}=\{\el_i\colon i \in I\} $  of $V$ we can write elements of $T(V)$ and $T((V))$ as finite linear combinations and formal series of elements in $\mathcal W (I)$, respectively, via the identification
\[i_1 \ldots  i_k= \el_{i_1} \ot\cdots \ot \el_{i_k}\,.\]
More generally, the concatenation operation is identified with the algebraic tensor product $\ot$, which forms an algebra on both $T(V)$ and $T((V))$. We define a bilinear and non-singular duality bracket $\langle \--, \-- \rangle: T((V)) \times T(V) \to \KK$ as
\begin{equation}\label{scalar}
 \left\langle \sum_{w \in \mathcal W(I)} \alpha_{w} w, v \right\rangle = \alpha_{v} \, , 
 \end{equation}
extended linearly to $T(V)$.
In this way, we can identify $ T((V))$ with the algebraic dual of $T(V)$.

In what follows, we will always suppose that all the preimages of the grading associated to $I$ are finite sets. Then $V$ becomes a graded vector space, and we can write $V = \oplus_{h \geq 0} V^h $ for the corresponding grading. The vector space $T(V)$ also inherits a grading using the weight of words by setting
\[T^h(V) \coloneqq \spn \{  w \in \mathcal W(I) \colon\,  ||w|| = h\} \,, \quad T(V)= \bigoplus_{h=0}^{\infty}T^h(V)\,, \]
\begin{equation*}
\begin{split}
T^{\leq h}(V)= \spn \{  w \in \mathcal W(I) \colon\,  ||w|| \leq  h\}\, ,  &\quad  T^{>h}(V)=\spn \{  w \in \mathcal W(I) \colon\, \, ||w|| >h\}\,.
\end{split}
\end{equation*}

The vector space $T^{\leq h}(V) \cong T(V)/{T^{>h}(V)}$ is called the space of \textbf{truncated tensors}. 
Since $T^{>h}(V)$ is a $\otimes$ ideal, the corresponding truncated tensor product is well defined on the quotient.
We write $\ot_h$ for the product on the quotient. We always have the isomorphism
$$T^{\leq h}(V)\cong T^{\leq h}\left(V^{\leq h}\right)\, ,$$
where $V^{\leq h} = \oplus_{k=0}^h V^k$. For $k \in \KK$, let $T_k^{\leq h}(V) \coloneqq \{ \vv \in T^{\leq h}(V) \colon  \langle \vv, \varepsilon \rangle = k\}$. These vector spaces are all finite-dimensional. The same truncation procedure applies to $T((V))$ obtaining a vector space which is isomorphic to $T^{\leq h}(V)$. For the sake of simplicity, we denote this vector space with the same notation $T^{\leq h}(V)$. This identification will also be used later in the context of Hopf algebras.

\subsection*{Shuffle and quasi-shuffle Hopf algebras}
Fix $d\geq 1$ integer, and consider henceforth $\KK[X_1, \dots, X_d]$ the vector space of polynomials over $d$ indeterminates.  Our vector space of interest will be defined simply by
\[V =\polyd =  \KK[X_1, \dots, X_d]/\{P\in \KK[X_1, \dots, X_d]\colon P=a,\; a\in \KK\}\] on which we study the tensor algebra and tensor series.

The vector space $\polyd$ has a basis of non-constant monomials. These are identified with $\MS_d$. We abuse notation and refer to a multiset $p \in \MS_d$ as a monomial in the variables $X_1, \dots, X_d$. 
For example, we identify the monomial $X_1^2X_2$ with the multiset $[112]$. This abuse of notation extends to the evaluation of a monomial, thus writing $[112](y_1, y_2, y_3) = y_1^2y_2$.
%To distinguish elements of the alphabet $\MS_d$ and scalars in $\mathbb{Z}$, we use typewritter typeset for multisets.
The alphabet $\MS_d$ is graded, with degree given by the polynomial degree of the associated monic polynomial, or equivalently by the set size. 
Given any $h\geq 0$ we use the notation $\polydh:=(\polyd)^{\leq h} $.

Two products can be defined on $T(\polyd)$: the \textbf{shuffle} $\shuffle$ and the \textbf{quasi-shuffle} $\qshuffle$ products. 
We define them here recursively.
For any $p,q \in \MS_d$ and $w,v\in \mathcal{W}(\MS_d)$, we set
\begin{equation}\label{shuffle_qshuffle}
\begin{split}
w =&\varepsilon \qshuffle w = w \qshuffle \varepsilon= \varepsilon \shuffle w = w \shuffle \varepsilon\, \\
(wp)\shuffle (v  q) =& (w \shuffle  v  q) p + (w  p\shuffle v)  q\\
(wp) \qshuffle (v  q) =& (w \qshuffle  v q) p + (w p \qshuffle v)  q+ (w \qshuffle v) [pq]\,,
\end{split}
\end{equation}
where  $[pq]$ is the contracted multi-set given by the union. These relations define two commutative algebras on $T(\polyd)$ which are compatible with the grading of $T(\polyd)$ given above, see \cite[Theorem 2.1]{hoffman2000} for a proof of this fact.

The tensor algebra $T(\polyd)$ can be further equipped with two structures of \textbf{Hopf algebras} by introducing the deconcatenation coproduct $\delta\colon T(\polyd) \to  T(\polyd) \otimes T(\polyd)$ and a  counit  $\eta^*\colon T(\polyd) \to \KK$.
For a word  $w = p_1  \cdots  p_k$, we set
\[\delta (\uw{w})=\uw{\varepsilon}\otimes \uw{w} + \uw{w} \otimes \uw{\varepsilon} + \sum_{l=1}^{k-1} \uw{p_1  \ldots  p_l} \otimes \uw{p_{l+1} \ldots p_k} ,\quad \eta^*(w): =\left\{
	\begin{array}{ll}
    1 & \mbox{if } w=\varepsilon\,, \\
	0 & \mbox{otherwise.}
	\end{array}
	\right.\]
We define as well the reduced coproduct map $\tilde{\delta} = \delta - \id\ot \varepsilon - \varepsilon \ot \id$. 
This endows $(T(\polyd),\shuffle, \delta)$ and   $(T(\polyd), \qshuffle, \delta)$ with graded Hopf algebra structures.
We expect there to be no confusion between elements in $T(V)$ and $T(V)\otimes T(V)$.

These Hopf algebras were shown to be isomorphic.
Explicit algebra morphisms $\Phi_H, \Psi_H:T(\polyd) \to T(\polyd)$ were constructed in \cite{hoffman2000}, which are inverses of each other.
Specifically, the maps $\Phi_H$ and $\Psi_H$ are the linear maps that act on words as follows:
\begin{align}\label{Hoff_exp_log}
\Phi_H(\uw{w}) &\coloneqq \sum_{\alpha \in C( |w |)} \frac{1}{\alpha !} \uw{(w)_{\alpha}}\,, \quad\quad
\Psi_H(\uw{w}) \coloneqq \sum_{\alpha \in C( |w |)} \frac{(-1)^{|w| - \ell(\alpha)}}{\Pi \alpha} \uw{(w)_{\alpha}}\,.
\end{align}
For instance, we have the following identities
\begin{align*}
\Phi_H( [1][2]) =& [1][2]+ \frac{1}{2} \uw{[12]} \,, \quad\quad
\Psi_H([1][2]) = [1][2] - \frac{1}{2} [12]\,, \\
\Phi_H([1][2][3]) =&[1][2][3] + \frac{1}{2} [12][3] + \frac{1}{2} [1][23] + \frac{1}{6}[123]\,,\\
\Psi_H([1][2][3]) =& [1][2][3] - \frac{1}{2} [12][3] - \frac{1}{2} [1][23] + \frac{1}{3}[123]\,.
\end{align*}
$\Phi_H$ is a graded isomorphism from $(T(\polyd),\shuffle, \delta)$ to $(T(\polyd),\qshuffle, \delta)$ and $\Psi_H= \Phi_H^{-1}$.

The adjoints of $\Phi_H, \Psi_H$ with respect to $\langle \--, \-- \rangle$ are also explicitly described in \cite[Section 4.2]{hoffman2000}. Specifically, $\Phi_H^*, \Psi_H^* : T((\polyd)) \to T((\polyd))$  are the maps defined by the following identities:
\begin{equation}\label{eq:phistar}
\begin{split}
\Phi^*_H(\uw{p_1 \ldots  p_k}) \coloneqq \Phi^*_H(\uw{p_1}) \ldots\Phi^*_H(\uw{p_k})\,,  \;\;
&\Psi^*_H(\uw{p_1 \ldots  p_k}) \coloneqq \Psi^*_H(\uw{p_1})\ldots   \Psi^*_H(\uw{p_k})\\
 \Phi^*_H(\uw{p}) \coloneqq \sum_{n \geq 1}\frac{1}{n!}\sum_{[p_1 \cdots p_n] = p} p_1 \ldots p_n\,,  \;\;
&\Psi^*_H(p) \coloneqq \sum_{n \geq 1}\frac{(-1)^{n-1}}{n}\sum_{[p_1 \cdots p_n] = p} \uw{p_1 \ldots  p_n}\, ,
\end{split}
\end{equation}
where the sum $[p_1 \cdots p_n] = p$ means that we sum over all the possible ways to write a letter $p$ as $[p_1 \ldots  p_n]$. For instance, one has
\begin{align*}
\Phi^*_H([12]) =& [12] + \frac{1}{2} ([1][2]+[1][2])\,, \quad \quad
\Psi^*_H([12]) = [12] - \frac{1}{2} ([1][2]+[1][2])\,, \\
\Phi^*_H([123]) =& [123] + \frac{1}{2} ([12][3]+ [1][23]+ [3][12]+ [23][1]+ [13][2]+ [2][13]) \\&+  \frac{1}{6}( [1][2][3]+ [2][1][3] + [3][2][1] + [2][3][1] + [1][3][2] + [3][1][2])\\
\Psi^*_H([123]) =&  [123] + \frac{1}{2} ([12]3+ [1][23]+ [3][12]+ [23][1]+ [13][2]+ [2][13]) \\&+  \frac{1}{3}( [1][2][3]+ [2][1][3] + [3][2][1] + [2][3][1] + [1][3][2] + [3][1][2])\,. 
\end{align*}
Let $H_{\shuffle}(\polyd), H_{\qshuffle}(\polyd)$ be the Hopf algebras dual to $(T(\polyd), \shuffle, \delta)$ and $(T(\polyd), \qshuffle, \delta)$ respectively. Note that the product structure of $H_{\shuffle}(\polyd), H_{\qshuffle}(\polyd)$ is $\ot$. By simple dualisation, the adjoint maps $\Phi_H^*, \Psi^*_H$ are \textbf{graded Hopf algebra isomorphism} between $H_{\shuffle}(\polyd)$ and $H_{\qshuffle}(\polyd)$, see \textit{e.g.} \cite[Proposition 3.5]{Bellingeri2022} for further explanations. 

The maps $\Phi_H, \Psi_H, \Phi^*_H$ and $\Psi^*_H$ are graded. Therefore, for any $h\geq 0$ these maps restrict to isomorphisms of vector spaces.
For instance, $\Phi^*_H$ is an isomorphism  between the algebras $(T^{\leq h}(\polyd),\otimes_h)$ and $(T^{\leq h}(\polyd),\otimes_h)$. We display these maps in \cref{cd:level_h_lie}, in the context of two important subspaces of $T((\polyd))$, that we introduce in the next section.

In what follows, it is also useful to introduce the following shorthand notation for the combinatorial coefficients 
\[n_{d,h}= \sum_{\alpha\in C(h)}\binom{d-1 +\alpha_1}{\alpha_1}\cdots\binom{d-1 +\alpha_{\ell(\alpha)}}{\alpha_{\ell(\alpha)}}\,.\]
As shown in \cite[Remark 2.3]{Tapia20}, these coefficients describe the dimension of the vector spaces  $T^h(\polyd)$,  the ambient space for the varieties we define below.

\section{Discrete signatures and their varieties \label{sec:computations}}
We now recall the definition of discrete signature from \cite{Tapia20}.%In what follows, we consider $\KK$ to be the real or the complex field of numbers.
\begin{defin}
 Fix integers $d, N\geq 1$ and let $x = (x_1, \dots , x_N)$ be a finite time series of elements in $\KK^d$. Denoting by $\Delta x_i = x_{i+1} - x_i \in \KK^d$ the increment time series, we define the \textbf{discrete signature} $\Dsign(x)\in T((\polyd))$ as the tensor series such that $\langle\Dsign(x), \varepsilon\rangle =1$ and
\begin{equation}\label{eq:dsign_unrep}
\langle \Dsign(x) , p_1 \ldots p_k \rangle \coloneqq \sum_{1\leq i_1 < \cdots < i_k < N} p_1(\Delta x_{i_1}) \cdots p_k(\Delta x_{i_k}) \, , 
\end{equation}
for any word $p_1 \ldots p_k \in \mathcal W  (\MS_d)$.
We use the convention that $\langle \Dsign(x), p_1 \ldots p_k \rangle  = 0$ if the summing set in \eqref{eq:dsign_unrep} is empty. 
Given $h\geq 0$, we denote the projection of  $\Dsign(x)$ onto $T^{\leq h}(\polyd)$ and $T^{h}(\polyd)$ by $\Dsign^{\leq h}(x)$ and $\Dsign^{h}(x)$,  referring to them as \textbf{truncated discrete signature of weight} $h$ and \textbf{discrete signature of weight} $h$.
\end{defin}

We note that $\langle\Dsign(x), w\rangle= 0$ whenever  $|w|\geq N$, therefore $\Dsign(x)$ and its projection on the vector space $\spn \{  w \in \mathcal W(\MS_d) \colon\,  |w| <  N\}$ are the same. 

\begin{smpl}\label{ex_1}
By identifying a time series $x = (x_1, \ldots , x_N)$  with a $d\times N$ matrix whose $i$-th column $x_i$ has coefficients $(x_i^1, \ldots , x_i^d)^T$, we can efficiently describe truncated signatures. For instance, one has 
\begin{align*}
&\Dsign^{\leq 2}\left(\begin{bmatrix}
1&2&3\\
2&3&2
\end{bmatrix}\right) =\\ =\;& \varepsilon  + (3-1)\, [1] + (2-2)\,[2] + ((2-1)^2+ (3-2)^2)\, [11] \\&+ ((3-2)^2+ (2-3)^2)\, [22] + ((2-1)(3-2)+(3-2)(2-3) )\, [12]\\ & + \, (2-1)(3-2)[1][1] + (2-1)(2-3)\,[1][2]\\& + (3-2)(3-2)\,[2][1] + (3-2)(2-3) [2][2] \\ = \;&\varepsilon + 2 \, [1] + 0\, [2] + 2\, [11] + 0 \, [12]+ 2\, [22]\\
& + \, [1][1] -\, [1][2] +\, [2][1] - [2][2]\, .
\end{align*}

More generally,  we have the formula
\begin{align*}
\Dsign^{\leq 2}\left(\begin{bmatrix}
x_1^1&x_2^1&x_3^1\\
x_1^2&x_2^2&x_3^2
\end{bmatrix}\right) =& \uw{\varepsilon} + (x_3^1-x_1^1) \, [1] + (x_3^2-x_1^2)\, [2] + \\&\, ((x_2^1-x_1^1)^2+(x_3^1-x_2^1)^2) \,[11]+ ((x_2^2-x_1^2)^2+(x_3^2-x_2^2)^2)\, [22]\\&+ ((x_2^2-x_1^2)(x_2^1-x_1^1)+(x_3^2-x_2^2)(x_3^1-x_2^1))\, [12] + \\
& + (x_2^1-x_1^1)(x_3^1-x_2^1)\, [1][1] + (x_2^1-x_1^1)(x_3^2-x_2^2)\, [1][2]\\&+ (x_2^2-x_1^2)(x_3^1-x_2^1)\, [2][1]+ (x_2^2-x_1^2)(x_3^2-x_2^2)\,[2][2] \, .
\end{align*}
\end{smpl}

\begin{smpl}[The Canonical Axis time series]
For any given $d\geq 1$ we set $N=d+1$ and we consider the following time series  $x^{\text{axis}}$   defined by $x^{\text{axis}}_1=0$  and the recursive condition $x^{\text{axis}}_{i+1}= x_{i}^{\text{axis}}+ e_i$
where $e_i$ for  $i=1, \ldots, d$ is the canonical basis  of $\R^d$. In few words,  $x^{\text{axis}}$ consists of $d$ steps from the zero vector $0$ to the unit vector $(1, \ldots, 1)^T$. Since the evaluation of a monomial $[j_1\ldots j_l]$ on the vector $e_i$ is non-zero if and only if $j_i= \ldots =j_l=i$,  the entry $\langle \Dsign(x^{\text{axis}}), w\rangle$  is zero unless the word $w$ has the form  $w=[i_1\ldots i_{1}]  [i_2\ldots i_2 ]\ldots  [ i_l\ldots i_l]$  with $i_1\leq i_2 \leq \ldots \leq  i_l$.  In that case, it equals to $1$. For instance, when $d=2$ one has 
\[\langle\Dsign(x^{\text{axis}}),[11]\rangle= \langle\Dsign(x^{\text{axis}}),[22]\rangle=\langle\Dsign(x^{\text{axis}}), [1][2]\rangle=1\,,\]
\[\langle\Dsign(x^{\text{axis}}),[12]\rangle= \langle\Dsign(x^{\text{axis}}), [2][1]\rangle=\langle\Dsign(x^{\text{axis}}), [2][2]\rangle= \langle\Dsign(x^{\text{axis}}), [1][1]\rangle=0\,,\]
as one can also check from Example \ref{ex_1}.
\end{smpl}

 \begin{rem}\label{rm_lite}
The notion of discrete signature presented here coincides with the evaluation of the \textbf{iterated-sums signature} $\text{ISS}(x)_{1, N}\in  T((\polyd))$ on its initial and terminal point.  For any given  infinite time series $x\colon \mathbb{N}\to \KK^d$, the tensor series $\text{ISS}(x)_{n,m}\in  T((\polyd))$ is defined by
\begin{equation}\label{eq:ISS}
\langle \text{ISS}_{n,m}(x) , p_1\ldots p_k \rangle \coloneqq \sum_{n\leq i_1 < \cdots < i_k < m} p_1(\Delta x_{i_1}) \cdots p_k(\Delta x_{i_k}) \, , 
\end{equation}
for any $n<m$, see \cite[Definition 3.1]{Tapia20}.
In particular, one has that $\text{ISS}(x)_{n,m} = \Dsign(x_n, x_{n+1}, \ldots, x_m)$.  
\end{rem}

Discrete signatures share some general algebraic identities among their coefficients, which justify using quasi-shuffle Hopf algebras. To state them precisely, we define the concatenation of time series.
\begin{defin}
Let  $ x=(x_1, \dots , x_N)$ and $ y=(y_1, \dots , y_M)$  be two time series, we define the \textbf{concatenated time series} $x|y\colon \{1, \ldots, N+M\}\to \KK^d$ by the conditions
\begin{equation}\label{def:concatenation}
(x|y)_k:=\begin{cases}x_k &k \in \{1, \ldots,N\}\\ y_{k-N}+x_N - y_1 &k \in \{N+1, \ldots,N+M\}\end{cases} 
\end{equation}
\end{defin}

In this way, if $x = \begin{bmatrix}
1&2&3\\
1&4&9
\end{bmatrix}$ and $y = \begin{bmatrix}
0&-1&1\\
2&3&5
\end{bmatrix}$, then 

\[(x | y ) = \begin{bmatrix}
1&2&3&3&2&4\\
1&4&9&9&10&12
\end{bmatrix}\,.\]

\begin{thm}\label{thm_prop_disc_sig}
The discrete signature has the following properties:

\begin{itemize}
\item  (Time-warping invariance) for any integer $M\geq 1$ one has 
\begin{equation}\label{eq_time_warp}
\Dsign (x_1, \dots , \underbrace{x_N, \dots, x_N}_{M \; \text{times}}) = \Dsign (x_1, \dots , x_N)\,.
\end{equation}

\item (Quasi-shuffle identity)
for any two words $w, v$ in $\MS_d$ we have
\begin{equation} \label{thm:qsrels}
\langle \Dsign (x), w \qshuffle v \rangle = \langle \Dsign (x), w \rangle \langle \Dsign (x), v \rangle \, .
\end{equation}
\item (Chen's identity) For any couple of time series  $ x=(x_1, \dots , x_N)$ and $ y=(y_1, \dots , y_M)$ one has 
\begin{equation}\label{eq_chen_equation}
\Dsign(x) \otimes\Dsign (y) = \Dsign(x|y)\,,
\end{equation}
\end{itemize}
\end{thm}

\begin{proof}
The first two properties follow directly from \cite[Section 4]{Tapia20} and \cite[Theorem 3.4, Part 1]{Tapia20}.  Concerning the Chen identity, this one follows directly from the Chen property on iterated-sums signatures (see \cite[Theorem 3.4, Part 2]{Tapia20}). Indeed, using the notation of Remark \ref{rm_lite} with this result, we have 
\[\Dsign (x|y)=\text{ISS}_{1,N+M}(x|y)=\text{ISS}_{1,N}(x|y) \otimes\text{ISS}_{N,N+M}(x|y)= \Dsign (x) \otimes \Dsign (y)\,. \]
\end{proof}

\begin{rem}
From this last identity in the proof, we understand why  in the definition \eqref{def:concatenation} we need to add the factor  $x_N-y_1$ to the part associated to $y$. Indeed,  when we evaluate both sides of the identity $\Dsign (x|y)= \Dsign (x) \otimes \Dsign (y)$ along the letter $w=[i]$, we obtain on  the right-hand  $x^{i}_N- x^{i}_1+y^{i}_M- y^{i}_1$ which corresponds to $(x|y)^i_{M+N}-(x|y)^i_{1}$. Without the translation this identity would be false.
\end{rem}
\begin{rem}\label{rem_trun_char}
By projecting the identity \eqref{eq_chen_equation} on $T^{\leq h}(\polyd)$ we also obtain the truncated identity
\begin{equation}\label{trunc_chen}
\Dsign^{\leq h}(x) \otimes_h\Dsign^{\leq h} (y) = \Dsign^{\leq h}(x|y)\,.
\end{equation}
Moreover  one has 
\[
\langle \Dsign^{\leq h} (x), w \qshuffle v \rangle = \langle \Dsign ^{\leq h} (x), w \rangle \langle \Dsign^{\leq h}  (x), v \rangle \, .
\]
for all words $w$ $v$ such that $||w|| + ||v||  \leq h$.
\end{rem}
We now introduce the related varieties associated with the discrete signature. For any  fixed integer $h\geq 1$ we consider the discrete signature of weight $h$ as a \textbf{polynomial mapping}  of the underlying time series $x$, i.e.
\[\Dsign^h\colon \KK^{d\times   N}\to T^h(\polyd)\simeq \KK^{n_{d,h}}\,.\]
Any coordinate of $\Dsign^h$ is a homogeneous polynomial of degree $h$. Therefore  $\Dsign^h$ induces a well-defined rational map among projective spaces 
\[(\Dsign^h)'\colon  \mathbb{P}^{dN-1}(\KK) \dashrightarrow \mathbb{P}^{n_{d,h}-1}(\KK). \]

This is a rational map, as there can be some values $\bar{x}\in \KK^{d \times N}\setminus \{0\}$  such that $\Dsign^h(\bar{x})=0 $. 
For any element $v$ of a vector space $V$, we denote by $[v]$ its equivalence class in the associated projective space. Every subset $U\subset V$  is associated to a  subset of the projective space $\mathbb{P}(V)$ by setting
\[[U]\coloneqq \{[u] \colon  u \in U\setminus \{0\}\}\,.\]

\begin{defin}
Fix $d, h, N$ integers $\geq 1$. We define the \textbf{affine discrete signature variety} $\V_{d, h, N}^{\text{af}}$ as the Zariski closure of the image of $\Dsign^h$ and the \textbf{discrete signature variety} is defined by setting $\V_{d, h, N}:=[\V_{d, h, N}^{\text{af}}]$. 
\end{defin}

For this paper, we focus on the case $\KK=\mathbb{C}$.

\begin{rem}\label{rk_par}
An alternative way to parametrise $\V_{d, h, N}$ is obtained by expressing $\Dsign$ using the \textbf{difference time series}. 
That is, for any $y\in \KK^{d(N-1)}$ and any word $w = p_1 \ldots  p_k$, consider the map $ \Dsign_{\Delta} $ such that:
\begin{equation}\label{eq:dsign_rep}
\langle \Dsign_{\Delta}(y) , \uw{w} \rangle \coloneqq \sum_{1 \leq i_1 < \cdots < i_k <N} p_1(y_{i_1}) \cdots p_k(y_{i_k}) \, . 
\end{equation}
Note that $\im \Dsign = \im \Dsign_{\Delta}$.
This reparametrisation allows for a more tractable computer-assisted calculation, as it reduces the input space of $\Dsign$ and simplifies implementation. 
We will use this parametrisation only in explicitly describing the varieties with weight $1$, $2$ and $3$.
\end{rem}

For any $h$ and $d$, let $N =1$ and note that $\V_{d, h, 1}= \emptyset$.  We see that our variety of interest behaves well with the variation of the index $N$.

\begin{thm}\label{thm:chain}
For any pair of integers  $1\leq N\leq M$ one  has $\V_{d, h, N}\subseteq \V_{d, h, M} $. Moreover, there exists an integer $N'$ depending on $d$ and $h$ such that the following ascending chain of varieties stabilizes.
\[
 \V_{d, h, 1}\subseteq \V_{d, h, 2} \subseteq \cdots  \subseteq  \V_{d, h, N'}= \V_{d, h, N'+1}= \cdots \, \, .
\]
We call $ \V_{d, h, N'}$ the 
  \textbf{universal discrete signature variety} and we denote it by $\mathcal{V}_{d,h}$.
\end{thm}

\begin{proof}
From the time-warping invariance equation in \eqref{eq_time_warp}, the image of the $\Dsign^h$ on time series in $\KK^{d\times   N}$ is included in the image of the $\Dsign^h$ on time series in $\KK^{d\times M}$. 
This concludes the chain claim. To note that the varieties stabilise, we use the classical fact that the vanishing ideal of a parametrised variety is a prime ideal.
Indeed, if $\phi: \KK [\mathbf{y} ] \to \KK [\mathbf{x}]$ is a parametrisation of a variety $V$, its vanishing ideal $I(V) = \ker \phi$ is the kernel of a ring homomorphism on an integral domain, so $I(V)$ is prime.

By considering the associated family of ideals, we obtain a descending chain of prime ideals in a polynomial ring on finite variables. This stabilises from Krull's principal ideal theorem.
\end{proof}

\begin{rem}
An estimation of the minimal $N'$ can be done, and the proof above gives an implicit upper bound, as any chain of prime ideals is bounded by the Krull dimension of a ring. (For polynomial rings over a field, it is well known that the Krull dimension of $\KK[X_1, \cdots, X_n]$ is $n$). Therefore, the decreasing chain of prime ideals
$$\mathfrak{p}_{d, h, 1} \supsetneq \mathfrak{p}_{d, h, 2} \supsetneq \cdots \supsetneq  \mathfrak{p}_{d, h, N'}\, ,$$
has at most $n_{d,h} $ terms.
\end{rem}
We describe some general facts about $\V_{d, h, N}$.

\subsection*{The weight one varieties}

For completeness' sake, we include this case here. For $h=1$ and any field $\KK$ of characteristic zero, the map $\Dsign^1$  is a linear map  with values in  $ T^1(\polyd)= \langle [1],\ldots,[d]\rangle $. In this case, for any $N\geq 2$, we have trivially  $\V^{\text{af}}_{d, 1, N}\cong \KK^d$ and therefore $\V_{d, 1, N} \cong  \mathbb{P}^{d-1}(\KK) $.

\subsection*{The weight two varieties}
We consider now discrete signature varieties of weight $2$ when $\KK= \mathbb{C}$. Firstly, using the definition of the ambient space $T^2(\polyd)$ in \cref{sec:prelim}, we can actually write it as the direct sum
\[T^2(\polyd)= \langle \{[i] [j] \colon 1\leq i\leq j\leq d \}\rangle\oplus \langle \{[ij]\colon 1\leq  i \leq j\leq d\}\rangle\,.\]
Therefore we can easily identify every element $ T^2(\polyd)$ as the direct sum of a  $d\times d$ matrix and a  symmetric $d\times d$ matrix. 
For any time series $x$ we can describe this decomposition with the coefficients
\[\Dsign_{[i][j]}\coloneqq \langle\Dsign(x),[i] [j]  \rangle\,, \quad \Dsign_{[ij] }\coloneqq\langle\Dsign(x),[ij]\rangle\,.\]

Thanks to the quasi-shuffle identities \eqref{thm:qsrels}, we deduce that these coefficients
satisfy the relation
\[\Dsign_{[i][j]}+\Dsign_{[j][i]}+\Dsign_{[ij]}=\langle\Dsign(x),[i] \rangle\langle\Dsign(x), [j] \rangle =(x_N^i- x_1^i) (x_N^j- x_1^i)\,.\]
Since the matrix with coefficients $(x_N^i- x_1^i) (x_N^j- x_1^j) $  has rank $\leq 1$, we obtain that $\V_{d, 2, N}$ is always included in the variety of the projective equivalence class of matrices $A$ of the form $A=S+M$ with $S$ symmetric and  $M$ generic such that 
\begin{equation}\label{eq:rk1cond}
\rk (S+ M+ M^\top)\leq 1.
\end{equation}
For instance, when $d=2$ this is equivalent to imposing that the matrix
$$\begin{bmatrix}
2\Dsign_{[1][1]} + \Dsign_{[11]} &\Dsign_{[12]} + \Dsign_{[1][2]}+\Dsign_{ [2][1]}\\
\Dsign_{ [12]} + \Dsign_{[1][2]}+\Dsign_{ [2][1]}& 2\Dsign_{[2][2]} + \Dsign_{[22]}
\end{bmatrix}\, .$$
is singular. In \cref{sec:conj} we will show that for $N$ sufficiently big the converse condition is also true.  That is, for any element $s \in T^2(\polyd)$ that satisfies \eqref{eq:rk1cond} there exists a time series $x$ such that $\Dsign^2(x) = s$.

\subsection*{Dimension one varieties}

We now look at the case where $d = 1$ and $h\geq 1$. In this setting, $\MS_1^0$ contains only multisets of the singleton $\{[1]\}$ and  $T^{h}(\polyd)$  has the  dimension $n_{1,h}=|C(h)|= 2^{h-1}$. From this equality, we can represent every $w  \in \mathcal{W}(\MS_1)$ such that $ \Vert w\Vert=h$ with a composition of $h$.   For instance, when $h=5$, we identify $w = [11][1][11]$ with the composition $(2, 1, 2)$ of $5$.

Using this notation, we can easily reinterpret the map $\Dsign^h$ using the theory of quasi-symmetric functions over the field $\KK$, see e.g. \cite{hazewinkel2001algebra}. A classic basis of this subring of formal power series in the  variables $y_{1},y_{2},y_{3},\ldots$ is given for any composition  $\alpha=(\alpha_1, \dots, \alpha_k)$ by the monomial formal series 
$$M_{(\alpha_1, \dots, \alpha_k)} = \sum_{1 \leq i_1 < 
\dots < i_k} y_{i_1}^{\alpha_1} \cdots y_{i_k}^{\alpha_k} \, . $$ 
Denoting by $M_w $ the monomial quasi-symmetric function indexed by the composition of $w\in \mathcal{W}(\MS_1)$, we immediately see that the coefficient  $\langle \Dsign^h (x), w \rangle $ is the evaluation of $M_w$ on $y=\Delta x$ appended with zeroes.

We recall that \cite[Theorem 8.1]{hazewinkel2001algebra} has shown that, the algebra of quasi-symmetric functions, is freely generated over the integers, and indeed gives a basis for this algebra, which is independent of the characteristic of the underlying field together with an explicit dimension computation. This result will be interpreted as a special case of  \cref{thm:dimension} where we compute the dimension a class of varieties containing any variety $\V_{1, h, N}$. 

\begin{smpl}
We analyse here further particular cases with $d = 1$ using the software   \texttt{Macaulay2} \cite{Mac2}.  

For $h = 3$ and $N= 4$, the signature variety is embedded in $\mathbb{P}^3(\mathbb{C})$. Using the parametrisation given in Remark \ref{rk_par}, the variety $ \mathcal{V}_{1,3,4}$  is associated to the rational map $\mathbb{P}^2(\mathbb{C}) \dashrightarrow \mathbb{P}^3(\mathbb{C})$ explicitly given by
$$ [y_1: y_2: y_3] \to [y_1  y_2   y_3: y_1^2  y_2+y_1^2  y_3+y_2^2 y_3: y_1  y_2^2 +y_1  y_3^2 +y_2  y_3^2: y_1^3+ y_2^3+ y_3^3] \, .$$
Using the notation $s_{(1,1,1)}, s_{(1,2)}, s_{(2,1)}, s_{3}$ for the coordinates of $\mathbb{C}^4$, we can see that the variety $\V_{1, 3, 3}$ is generated by one  equation of degree nine. Here is an excerpt of this equation:
\begin{align*}
81 &s_{(1,1,1)}^9+162 s_{(1,1,1)}^8 s_{(1,2)}+351 s_{(1,1,1)}^7 s_{(1,2)}^2+333 s_{(1,1,1)}^6 s_{(1,2)}^3\\&+72 s_{(1,1,1)}^5 s_{(1,2)}^4-63 s_{(1,1,1)}^4 s_{(1,2)}^5-30 s_{(1,1,1)}^3 s_{(1,2)}^6\\&+6 s_{(1,1,1)}^2 s_{(1,2)}^7+
6 s_{(1,1,1)} s_{(1,2)}^8+s_{(1,2)}^9+162 s_{(1,1,1)}^8 s_{(2,1)}-30 s_{(1,1,1)}^2 s_{(1,2)} s_{(2,1)}^6\\&+ \cdots +4 s_{(1,1,1)}^3 s_{(2,1)}^2 s_{3}^4+4 s_{(1,1,1)}^2 s_{(1,2)} s_{(2,1)}^2 s_{3}^4+s_{(1,1,1)} s_{(1,2)}^2 s_{(2,1)}^2 s_{3}^4\, .
\end{align*}
The dimension of the resulting variety $\mathcal{V}_{1,3,4}$ is two.
%\carlo{check again the computations (there was a mistake in the polynomial map)}

For $h = 4$ and $N= 4$, the variety $\V_{1, 4, 4}$ is embedded in $\mathbb{P}^7(\mathbb{C})$ and is generated by $20$ polynomials, whose degree counts are the following:

\begin{center}
\begin{tabular}{l|c|c|c|c}
Degree & 1 & 2 & 3 & 4\\
Quantity &1& 1& 12& 6
\end{tabular}
\end{center}

Note that the linear polynomial arises from the trivial relation 
$$M_{(1, 1, 1, 1)} ( \Delta x, 0, \ldots ) = 0.$$

\end{smpl}
\subsection*{Three steps}
Let us now focus on the case $N = 4$, $d = 2$ and $h = 3$. Using the parametrisation given in Remark \ref{rk_par}, the variety $ \mathcal{V}_{2,3,3}$, we rename some coordinates of $T^3(\polyd)$ as follows:
\begin{align*}
s_{i,j,k}=\langle \Dsign_{\Delta} (y), [i] [j] [k] \rangle\,,& \quad t_{i,j,k}= \langle \Dsign_{\Delta} (y), [ij][k]\rangle\,, \\ u_{i,j,k}= \langle\Dsign_{\Delta}(y), [i][jk]\rangle\,, & \quad v_{i,j,k}=\langle \Dsign_{\Delta} (y), [ijk]\rangle\,.
\end{align*}

Starting from a difference time series \[y = \begin{bmatrix}
a_{1} & b_{1} & c_{1}\\
a_{2} & b_{2} & c_{2}
\end{bmatrix}\,,\] we have the following parametric equations of $\mathcal{V}_{2,3,3}$:
\begin{align*}
s_{i, j, k} =& a_ib_jc_k\\
t_{i, j, k} =& a_ia_jb_k+ a_ia_jc_k+ b_ib_jc_k\\
u_{i, j, k} =& a_ib_jb_k+ a_ic_jc_k+ b_ic_jc_k\\
v_{i, j, k} =& a_ia_ja_k+ b_ib_jb_k+ c_ic_jc_k\, .
\end{align*}

There are 226 minimal generators of the ideal associated to $\mathcal \V_{2, 3, 4}$ of degree at most four. These break down into $58$ quadrics, $74$ cubic and $134$ quartic generators. Here is one of the cubics generating $\mathcal I(\V_{2, 3, 3})$.
\begin{align*}
&s_{121}s_{222}v_{222} - s_{121}t_{222}u_{222}-s_{122}s_{222}v_{122} + s_{122}t_{222}u_{212}\\
&-s_{221}s_{222}v_{122} + s_{221}t_{122}u_{222}+s_{222}^2v_{112}-s_{222}t_{122}u_{212}\, .
\end{align*}

\subsection*{Overview table}

We sum up some dimensions (of the homogeneous ideal), and degrees and also display some of the generators obtained with the software \texttt{Macaulay2}, using the code in \cite{M2_MATHREPO}. We mention that a general description of the varieties $\V_{d, h, N}$ using only computational tools is a computationally hard question, as even for small sizes like in $\V_{2, 2, 4}$, computing a Gr\"obner basis is too expensive.

For the number of generators, we do guarantee that this is the smallest possible number, as this is obtained with \texttt{Macaulay2} using the \textit{mingens} function, which is known to give an optimal set for homogeneous ideals.
Indeed, we construct the variety $\V_{d, h, N}$ in the extended ring with the free variables (corresponding to entries of the matrix), and compute its Gr\"obner basis $\mathcal G$.
Via the \cite[Implicitization theorem in Section 4.2]{michalek2021invitation}, the Gr\"obner basis of the original ideal in the ambient ring of $\V_{d, h, N}$ is a subset of $\mathcal G$, precisely those polynomials that do not have an occurrence of the free variables.
We use this fact in computing a generating set of the ideal.

\begin{center}
\begin{tabular}{|l| c | c | c | l|}
\hline
Variety & Dim & Deg & \# gens & Some generator\\
\hline
\hline
$\V_{1, 3, 3}$ & $2$ & $3^3$ & $1$ & $ 81s_{00}^9+162s_{00}^8s_{01}+351s_{00}^7s_{01}^2+333s_{00}^6s_{01}^3 \cdots  $\\
\hline
$\V_{1, 4, 3}$ & $2$ & $2^{4}$ & $20$ & $ 22050 s_{001} s_{100}^3 + 28284 s_{011} s_{100}^2 s_{101} + 257246 s_{100}^3 s_{101} $\\
& & & & \quad $  + 110804 s_{001} s_{010} s_{101}^2 + 89216 s_{010}^2 s_{101} ^2 + \cdots$\\
\hline
$\V_{2, 2, 3}$ & $5$ & $2 $ & $1$ & $s_{12}^2+2s_{12}s_{21}+s_{21}^2-4s_{11}s_{22}-2s_{22}t_{11}$\\
$ = \V_{2, 2, 4}$& & & & \quad $ -2s_{11}t_{22} -t_{11}t_{22}+2s_{12}t_{12}+2s_{21}t_{12}+t_{12}^2 $\\
\hline
$\V_{2, 3, 2}$ & $3$ & $3^{3}$ & $97$ & $4 t_{111}^2 u_{112} v_{112}^2 -t_{112} v_{111}^2 v_{112}^2+2 u_{111}^2 v_{112}^3$\\
%$\V_{2, 3, 3}$ & $7$ & $28$ & $...er$ & $ s_{12}^2+2s_{12}s_{21}+s_{21}^2-4s_{11}s_{22}-2s_{22}t_{11} $ \\
 & & & & \quad $-t_{111} v_{111} v_{112}^3 - 4 t_{111}^2 u_{112} v_{111} v_{122}+ \cdots$ \\
\hline
$\V_{3, 2, 2}$ & $5$ & $2^{5}$ & $45$ & $t_{23}t_{32}-t_{22}t_{33}, t_{13}t_{32}-t_{12}t_{33}, \cdots $ \\
\hline
\end{tabular}
\end{center}

\begin{rem}\label{rem:power_degree}
We remark that the degree of every variety presented here is a perfect power of a prime.
\end{rem}

\subsection*{Some generators of a high dimension case}

For $\mathcal{V}_{2,3, 4}$ it is not computationally feasible to present a  Gr\"obner basis, but we can perform the Buchberger algorithm and present some low degree terms generating the varieties that are presented in \cite{M2_MATHREPO}.

\section{Lie polynomials and Lie groups}\label{sec_Lie}
In this section, we introduce some fundamental subsets of $T^{\leq h}(\polyd)$: the space of Lie polynomials and the variety of group-like elements in the shuffle and quasi-shuffle context. These will play a role in understanding discrete signatures and, as the name suggests, they will form pairs of Lie algebra and Lie group. Most of the results in this section will follow from general facts on the theory of free Lie algebras \cite{Reutenauer89} which will be adapted to take into account our intrinsic notion of weight of words. The underlying field $\KK$ will just be of characteristic zero.

We define the \textbf{truncated Lie bracket} $[ \,,  ]_h : T^{\leq h}(\polyd) \otimes T^{\leq h}(\polyd) \to T^{\leq h}(\polyd)$ by setting 
 \[[v, w]_h \coloneqq w\otimes_h v - v\otimes_h w\,,\]
and extending it linearly. 
We define  $\mathcal L^{\leq h}(\polyd)$ as the Lie subalgebra of $T^{\leq h}(\polyd)$ generated by  $\polyd^{\leq h}$.  Equivalently, $\mathcal L^{\leq h}(\polyd)$ is the space of iterated Lie brackets starting from the finite-dimensional vector space $\polyd^{\leq h}$. We refer to this Lie algebra as the set of \textbf{weight-$h$ Lie polynomials}. By construction, this Lie algebra will be finitely generated by an explicit family of elements.   In what follows, we introduce for any $w\in  \mathcal{W}(\MS_d)$ the family of Lie polynomials $\lv_{w}\in T((\polyd))$ defined by $\lv_{\varepsilon} = 0$ and  the  recursive condition
 \begin{equation}\label{eq:prims}
 \lv_{w} = \lb p_1, \lv_{u}\rb\,, 
 \end{equation}
 where $w= p_1u$ and $\lb\,, \rb$ is  the Lie bracket without truncation on $T((\polyd))$
\[[\![v, w]\!] \coloneqq w v - vw\,.\]
To describe $\mathcal L^{\leq h}(\polyd)$, we introduce an intermediary Lie algebra that sits between $\mathcal L^{\leq h}(\polyd)$ and $T((\polyd))$. We denote by $\mathfrak{L}^{\leq h}$ the free, nilpotent Lie algebra of
step $h$ over $\polydh$. More explicitly, one has 
 \[\mathfrak{L}^{\leq h}= \polydh\oplus [\![\polydh, \polydh\rb \oplus\cdots \oplus \underbrace{ \lb \polydh ,\lb \polydh \ldots, \lb\polydh ,  \polydh \rb}_{\text{$h-1$ times}}\, ,\]
see  e.g. \cite[Definition 7.25]{frizbook}. Using $\mathfrak{L}^{\leq h}$, we can actually write $\mathcal L^{\leq h}(\polyd)$ as a quotient.
\begin{lm}\label{prop_quotient}
The space $\mathcal L^{\leq h}(\polyd)$ is a Lie algebra which arises as the quotient 
\begin{equation}\label{trunc_pol}
\mathcal L^{\leq h}(\polyd)=\mathfrak{L}^{ \leq h}/ \left(T^{>h}(\polyd )\cap\mathfrak{L}^{ \leq h} \right)\,.
\end{equation}
\end{lm}

\begin{proof}
By definition of $\otimes_h$,  the Lie product $[\,,]_h$ is $h$-step nilpotent. Therefore we can use the fundamental property of $\mathfrak{L}^{\leq h}$ as the free $h$-step nilpotent Lie algebra (see \cite[Remark 7.26]{frizbook}). This gives us an explicit Lie algebra morphism $\mathfrak{i}\colon \mathfrak{L}^{\leq h}\to \mathcal L^{\leq h}(\polyd)$ lifting the identity map on $\polydh $ and sending iterated $\lb\,, \rb$ Lie polynomials into $[\,,]_h$ polynomials. By definition of  $\mathcal L^{\leq h}(\polyd)$, the map $\mathfrak{i}$ is surjective  and one can see that its kernel is $T^{>h}(\polyd )\cap \mathfrak{L}^{\leq h} $. 
\end{proof}

  From the definition of $\mathfrak{L}^{\leq h}$ one has immediately that $\{ \lv_w : ||w|| \leq h\}$ is linear generating set for $\mathcal L^{\leq h}(\polyd)$. Additional properties of $\mathfrak{L}^{\leq h}$ allow us to write down an explicit basis for $\mathcal L^{\leq h}(\polyd)$. Using the total order on $\MS_d$,  we can order all words $ \mathcal{W} (\MS_d)$ with lexicographic order $\geq_{lex}$. A word $I\in \mathcal{W} (\MS_d)$ is said a \textbf{Lyndon word} if it is strictly smaller than all of its rotations. For any Lyndon word $I$, we can uniquely associate an iterated Lie bracket $b(I)\in T((\polyd))$ called \textbf{Lyndon bracket}. The definition of $b(I)$ follows by induction on $|I|$ by setting $b(I)=I$ if $|I|=1$ and  $b(I)= \lb b(I_1), b(I_2)\rb$ where $I= I_1 I_2$ and $I_2$ is the longest Lyndon word appearing as a proper right factor of $I$.

\begin{lm}\label{dim_lie}
The family of Lyndon brackets $b(I)$ where $I$ is a Lyndon word of weight smaller than $h$ is a basis for $\mathcal L^{\leq h}(\polyd) $. The dimension of this Lie algebra equal to 
\begin{equation}\label{eq:Lambda}
\Lambda_{d,h}:=\sum_{l=1}^h\lambda_{d,l}\,,
\end{equation}
where $\lambda_{d,l}$ denotes the number of Lyndon words in $\mathcal W (\MS_d)$ of weight $l$.
\end{lm}
\begin{proof}
Following \cite[Corollary 4.14]{reutenauer1993free},  we know that the family of Lyndon brackets $b(I)$ where $I$ is a Lyndon word of length smaller than $h$  is a basis for $\mathfrak{L}^{\leq h}$. Since any Lyndon bracket $b(I)$ preserves the weight of the underlying word $I$, the result follows from the quotient \eqref{trunc_pol}.
\end{proof}

The modification of the usual grading on word algebras with the introduction of a degree function is uncommon in the study of the coefficients $\lambda_{d,h}$. For instance, Lyndon words arise in the study of words on the alphabet $\{1, \dots, d\}$, taken with the degree function constant equal to one. There, words of length $h$ correspond to words of weight $h$. In that context, the number of Lyndon words of length $h$ is given  by
$$ \frac{1}{h}\sum_{k | h} \mu \left(k\right)d^{h/k} \, ,  $$
where $\mu $ is the M\"obius function on integers, see \cite[Section 0]{reutenauer1993free}. This reflects the fact that the weight and the length play the same role in this alphabet. 

In our context, we do not have the same weight and length as most words, but we can still derive a general formula.
\begin{thm}\label{thm:enum}
The number of Lyndon words in $\mathcal W (\MS_d)$ of weight $h$ is given by
\begin{equation}\label{eq_lyndon_words-weight}
\lambda_{d, h} =  \sum_{k|h} \frac{\mu\left(k\right)}{k} \sum_{\alpha\in C(h/k)} \frac{1}{\ell (\alpha)} \prod_{i=1}^{\ell(\alpha)} \binom{\alpha_i + d - 1}{d - 1} \, . 
\end{equation}
\end{thm}
\begin{proof}
First we note that the weight grading on $T(\polyd)$ gives us the following power series
$$H(x) = \sum_{w \in \mathcal W(\MS_d)} x^{||w|| } = \frac{1}{1 - \sum_{I \in \MS_d} x^{\text{grad}(I)}} \, ,$$
where $\text{grad}(I)$ is the cardinality of multiset $I$. Since for any $h \geq 1 $ there are $\binom{h+d-1}{d-1}$ multisets on $\{1, \dots, d\}$ of cardinality $h$ one has
$$ \sum_{I \in \MS_d} x^{\text{grad}(I)} = \sum_{k\geq 1} \binom{k + d -1 }{d-1} x^ k = (1 - x)^{-d} - 1\, ,$$
thereby obtaining $H(x) = [1 - ( (1-x)^{-d} - 1)]^{-1}$.

On the other hand, the \textbf{Lyndon unique factorization theorem} (see \cite{chen1958free}) guarantees that each word $w \in \mathcal W (\MS_d)$ can be written uniquely as $ w = \tau_1  \ldots  \tau_j \, , $ where $\tau_i $ are Lyndon words with $\tau_1 \geq_{lex} \cdots \geq_{lex} \tau_k$. 
Therefore
\begin{align*}
H(x) &= \sum_{w \in \mathcal W(\MS_d)} x^{||w|| } = \prod_{\substack{\tau \text{ Lyndon word} \\ \tau  \in \mathcal W(\MS_d )}}\left( 1 + x^{||\tau||} + x^{2||\tau||} + x^{3||\tau||} + \cdots  \right)  \\
&= \prod_{\substack{\tau \text{ Lyndon word} \\ \tau  \in \mathcal W(\MS_d)}} ( 1 - x^{||\tau||} )^{-1} = \prod_{h\geq 1} ( 1 - x^{h} )^{-\lambda_{d, h}},
\end{align*}
where we recall that $\lambda_{d,h} $ is the number of Lyndon words $\tau \in \mathcal W(\MS_d)$ of weight $h$. Putting it together, applying $\log $ on both sides and using  the expansion
\[-\log(1-f(x) ) = \sum_{h\geq 1} \frac{1}{h}f(x)^ h\,,\] 
we get the equivalent series expansions
\begin{align*}
-\log(H(x) ) &= \sum_{h\geq 1} - \lambda_{d, h} \log ( 1 - x^{h} )  = \sum_{j, h\geq 1}\frac{1}{j} x^{j h} \lambda_{d, h}  = \\
&= \sum_{h\geq 1} x^h \sum_{k | h} \frac{1}{h/k} \lambda_{d, k}  = \sum_{h\geq 1} \frac{1}{h}x^h \sum_{k |h} k \lambda_{d, k} \,,\\
\\
-\log(H(x) ) &= -\log [ 1 - ( (1-x)^{-d} - 1 ) ]  = \\
&= \sum_{k\geq 1} \frac{1}{k}\left(\sum_{t\geq 1} \binom{t + d -1 }{d-1} x^t \right)^k\\
&= \sum_{h\geq 1} x^h\sum_{\alpha\in C(h)} \frac{1}{\ell(\alpha)}\prod_{i=1}^{\ell(\alpha)} \binom{\alpha_i + d - 1}{d-1}\,.
\end{align*}

Equating both sides, it follows that for each $h\geq 1$ we have
$$ \sum_{k | h} k \lambda_{d, k} =   h \sum_{\alpha\in C(h)} \frac{1}{\ell(\alpha)}\prod_{i=1}^{\ell(\alpha)} \binom{\alpha_i + d - 1}{d-1} \, .$$
By applying M\"obius inversion formula we get 
$$h \lambda_{d, h} = \sum_{n|h} \frac{h}{n}\, \, \mu\left( n \right) \sum_{\alpha\in C(h/n)} \frac{1}{\ell(\alpha)}\prod_{i=1}^{\ell(\alpha)}  \binom{\alpha_i + d - 1}{d-1}\, ,$$
from which the theorem follows.
\end{proof}

\begin{rem}
It is somewhat surprising that the formula \eqref{eq_lyndon_words-weight} always yields integers values. 
It is a corollary of the proof above that the intermediary terms
$$k \sum_{\alpha\in C(k)} \frac{1}{\ell (\alpha)} \prod_i \binom{\alpha_i + d - 1}{d - 1} $$
are also integers.
\end{rem}

This allows us to create the following values of $\lambda_{d, h} $ in \cref{tbl:dimension}.  

\begin{table}
\begin{center}
\begin{tabular}{|l|| c | c | c |c |c |c|c | c | c |}
\hline
$h$ & 1 & 2 & 3 & 4 & 5 & 6 & 7 & 8 &9 \\
\hline
$d = 1$&  1& 2& 2& 3& 6& 9& 18& 30& 56\\
$d = 2$&  2& 4& 12& 32& 92& 256& 772& 2291& 7000 \\
$d = 3$&  3 & 8& 36& 132& 534& 2140& 8982& 38031& 164150 \\
$d = 4$&  4& 16& 80& 380& 1960& 10228& 55352& 304223& 1700712 \\
$d = 5$&  5& 26& 150& 875& 5500& 35335& 234530& 1584845& 10885640 \\
$d = 6$&  6& 36& 252& 1743& 12936& 98686& 776412& 6226008& 50732712  \\
$d = 7$&  7& 48& 392& 3136& 26852& 237160& 2158156& 20028764& 188856934   \\
\hline
\end{tabular}
\end{center}
\caption{Dimension of the Lie algebra $\mathcal L^{\leq h}(\polyd)$  \label{tbl:dimension}}
\end{table}

To pass from Lie algebras to Lie groups, we introduce the polynomial maps
\[\exp_{\otimes_h}\colon T^{\leq h}_0(\polyd) \to T^{\leq h}_1(\polyd) \, \quad \text{and} \quad \log_{\otimes_h}:  T^{\leq h}_1(\polyd ) \to T^{\leq h}_0(\polyd )\] defined by
\begin{equation}\label{exp_log_bullet}
\exp_{\otimes_h}(\vv)\coloneqq\sum_{n \geq 0 } \frac{1}{n!} \vv^{\otimes_h \,  n} \,, \quad \log_{\otimes_h}( v) \coloneqq \sum_{n\geq 0} \frac{(-1)^{n-1}}{n} (\vv- \varepsilon)^{ \otimes_h\, n }\, ,
\end{equation}
where $\vv^{\otimes_h \, n} = \vv \otimes_h \cdots \otimes_h \vv$ stands for the $n$-th truncated tensor product. It is known in the literature, see e.g. \cite[Chapter 3]{reutenauer1993free},  that the equivalent definition of $\exp_{\otimes_h}$ and $\log_{\otimes_h}$ with the product $\otimes$ are well defined map from $T_0((\polyd ))$ to $T_1((\polyd ))$ with the log being the inverse of the exponential. By simple quotienting, we obtain that  $\exp_{\otimes_h}$ and $\log_{\otimes_h}$ are well defined maps such that  $\log_{\otimes_h}= \exp_{\otimes_h}^{-1}$.

We define the \textbf{weight-$h$ free nilpotent Lie group} $\mathcal{G}^{\leq h}(\polyd)$ as the image
\[\mathcal{G}^{\leq h}(\polyd)\coloneqq\exp_{\otimes_h}(\mathcal L^{\leq h}(\polyd))\subset T^{\leq h}_1(\polyd)\,. \]
Similarly as $\mathcal{L}^{\leq h}(\polyd)$, we can indeed describe $\mathcal{G}^{\leq h}(\polyd)$ as an explicit quotient of a more known Lie group. We denote by  $\mathfrak{G}^{\leq h}$ the free nilpotent group of step $h$ over $\polydh$, see \cite[Theorem 7.30]{frizbook}, which is defined as
\[\mathfrak{G}^{\leq h}= \exp_h( \mathfrak{L}^{\leq h})\]
where $\exp_h\colon T((\polyd))\to T((\polyd))$ is defined by
$\exp_h(\vv)\coloneqq\sum_{n \geq 0 }^h \frac{1}{n!} \vv^{\otimes \,  n}\,$, where $\vv^{\otimes \, n} = \vv \otimes \cdots \otimes \vv$.

\begin{lm}
$\mathcal{G}^{\leq h}(\polyd)$ with the operation $\otimes_h$ is a Lie group  which arises as the quotient 
\begin{equation}\label{trunc_pol2}
\mathcal G^{\leq h}(\polyd)=\mathfrak{G}^{\leq h}/ \left(T^{>h}(\polyd )\cap\mathfrak{G}^{\leq h}\right)\,.
\end{equation}
\end{lm}
\begin{proof}
Since we know already
 from the literature that $\mathfrak{G}^{\leq h}$ is a simply connected Lie group with Lie algebra equal to $\mathfrak{L}^{\leq h}$, see \cite[Theorem 7.30]{frizbook}, using the identification in \eqref{trunc_pol}, it is sufficient to check   that the following diagram
\begin{center}
\begin{tikzcd}
\mathfrak{L}^{\leq h} \arrow{d}{\pi} \arrow{r}{\exp_h}
& \mathfrak{G}^{\leq h} \arrow{d}{\pi} \\
\mathfrak{L}^{\leq h}/ \left(T^{>h}(\polyd )\cap\mathfrak{L}^{\leq h} \right) \arrow{r}{\exp_{\otimes_h}}
& \mathfrak{G}^{\leq h}/ \left(T^{>h}(\polyd )\cap\mathfrak{G}^{\leq h} \right)
\end{tikzcd}
\end{center}
commutes, where we denote by $\pi$ the projections on the quotient. This follows trivially from simple consideration regarding  the projection on $T^{>h}(\polyd )$.
\end{proof}

Thanks to these identifications, we can transfer the fundamental  properties of $ \mathfrak{G}^{\leq h}$ and $\mathfrak{L}^{\leq h}$ to $\mathcal{G}^{\leq h}(\polyd)$ and $\mathcal{L}^{\leq h}(\polyd)$. The first one is the Chen-Chow theorem, which expresses elements of $\mathcal G^{\leq h}(\polyd )$ as concatenation of simpler element. See \cite[Theorem 7.28]{frizbook} for the proof on $\mathfrak{G}^{\leq h}$.
\begin{thm}\label{thm:chow}
For any integer $h\geq 1$ and $g \in \mathcal G^{\leq h}(\polyd ) $ there exists an integer $m$ and $\vv_1, \ldots , \vv_m \in  \polydh $ such that 
$$ g = \exp_{\otimes_h}(\vv_1) \otimes_h \cdots \otimes_h \exp_{\otimes_h}(\vv_m) \, . $$
\end{thm}

The second property we mention is also related to a representation of $\mathcal G^{\leq h}(\polyd )$ as an  affine variety. This result follows by applying the classical properties on  free nilpoltent Lie algebras and then the quotient. See \cite[Theorem 3.1, Theorem 3.2]{reutenauer1993free} and  \cite[Lemma 4.1, Lemma 4.2]{amendola2019varieties}.

\begin{thm}\label{shuffle_Lie}
The elements of $\mathcal G^{\leq h}(\polyd )$ and $\mathcal L^{\leq h}(\polyd)$ are characterized by the following relations
\begin{equation}\label{lie_shuffle}
\mathcal L^{\leq h}(\polyd)=\{\vv\in T^{\leq h}_0(\polyd) \colon \quad \langle \vv,u\shuffle k\rangle=0 \quad \text{for all} \;||u|| + ||k||  \leq h \}\,.
\end{equation}
\begin{equation}\label{trunc_groups}
\mathcal G^{\leq h}(\polyd ) =\left\{ \vv \in T^{\leq h}_1(\polyd ) \colon \, \, 
\langle \vv, \uw{w }\shuffle \uw{u }\rangle = \langle \vv, \uw{w }\rangle \langle \vv,\uw{ u}\rangle   \, \, \text{for} \;||w|| + ||u||  \leq h \right\}\,. 
\end{equation}
\end{thm}

We now turn our attention to the quasi-shuffle relations and define an analogous space to the one presented in \cref{shuffle_Lie}, in the quasi-shuffle context.

\begin{defin}\label{def_quasi-shuffle}
For any integer  $h\geq 1$ we define the \textbf{weight-$h$ free quasi-shuffle Lie group} $\hat{\mathcal G}^{\leq h}(\polyd )$ and  \textbf{weight-$h$ quasi-shuffle Lie polynomials} $\hat{\mathcal{L}}^{\leq h}(\polyd)$ as 
\[
\hat{\mathcal G}^{\leq h}(\polyd ) \coloneqq \left\{ \vv \in T^{\leq h}_1(\polyd ) \colon \, \, 
\langle \vv, \uw{w }\qshuffle \uw{u }\rangle = \langle \vv, \uw{w }\rangle \langle \vv,\uw{ u }\rangle   \, \, \text{for all} \;||w|| + || u||  \leq h \right\}\,,\]
\[\hat{\mathcal{L}}^{\leq h}(\polyd)\coloneqq\log_{\otimes_h}(\hat{\mathcal G}^{\leq h}(\polyd ))\subset T^{\leq h}_0(\polyd)\,.\]
\end{defin}

Recall that $\Phi_H$ and $ \Psi_H$, defined in \eqref{Hoff_exp_log}, and their adjoints, are Hopf algebra isomorphisms. Summing up the results in \cite[Theorem 4.2]{hoffman2000} and \cite{Bellingeri2022} one has the following properties:

\begin{thm}\label{prop:phi_iso}
$\hat{\mathcal G}^{\leq h}(\polyd )$ with the operation $\otimes_h$ is a Lie group with $\hat{\mathcal{L}}^{\leq h}(\polyd)$ as  Lie algebra. Moreover, the function $\Phi^*_H$ maps isomorphically $\hat{\mathcal G}^{\leq h}(\polyd) $ to $\mathcal G^{\leq h}(\polyd )$ and $\hat{\mathcal L}^{\leq h}(\polyd) $ to $\mathcal L^{\leq h}(\polyd )$.
The maps $\Phi^*_H, \Psi_H^*$ commute with $\exp_{\otimes_h}, \log_{\otimes_h} $ on these domains.
\end{thm}

The importance of this last Lie group/ Lie algebra pair is important because the truncated discrete signature of a time series is an element of $\hat{\mathcal G}^{\leq h}(\polyd)$, according to \cref{rem_trun_char}. Summing up the relation in a commutative diagram, we can describe the properties of $\Phi^*_H$ and $\Psi^*_H$ in the following diagram

%\vspace*{-0.5cm}

\begin{equation}\label{cd:level_h_lie}
\begin{tikzcd}
& & \hat{\mathcal{L}}^{\leq h}(\polyd) \arrow[rr, bend left=15, "\Phi^{ *}_H" description] \arrow[dd, bend left=30, "\exp_{\otimes_h}" description] & & \mathcal{L}^{\leq h}(\polyd ) \arrow[ll, bend left=15, "\Psi^{*}_H" description] \arrow[dd, bend left=30, "\exp_{\otimes_h}" description] \\ \KK^{d\times N} \arrow[rrd, bend right=15, "\mathscr{S}^{\leq h}" description] & & & &\\
& & \hat{\mathcal G}^{\leq h}(\polyd )  \arrow[rr, bend left=15, "\Phi^{*}_H" description] \arrow[uu, bend left=30, "\log_{\otimes_h} " description] & & \mathcal G^{\leq h}(\polyd )  \arrow[ll, bend left=15, "\Psi^{*}_H" description] \arrow[uu, bend left=30, "\log_{\otimes_h}" description] 
\end{tikzcd}
\end{equation}

%\caption{The weight $h$ free Lie algebras and the weight $h$ Lie groups are connected via $\Phi^{*}$ and $\Psi^{*}$.\label{cd:level_h_lie}}

From an algebraic geometry point of view, we can reinterpret $\hat{\mathcal G}^{\leq h}(\polyd )$ as an affine variety. Using the notation $\mathcal R_{w} \in \KK $ for a word $w\in \mathcal{W}(\MS_d)$ with $||w||\leq h$ to denote the coordinates of $w$ in $T^{\leq h}(\polyd)$, we consider the following ideal of polynomial functions on $T^{\leq h}(\polyd)$
\begin{equation}\label{ideal_affine}
\hat{G}_{d,\leq h}\coloneqq \langle \mathcal R_{w}\mathcal R_{v}- \mathcal R_{w\qshuffle v}\;\colon \text{for all words  $w$ and $v$ s.t.} \;||w|| + || v||  \leq h\rangle\,.
\end{equation}
The zero set of this ideal is exactly $\hat{\mathcal G}^{\leq h}(\polyd )$. Combining the Lie group structure with this affine representation we can indeed give a first description of this variety by computing its dimension.

\begin{thm}\label{thm_dim_1}
The ideal $\hat{G}_{d,\leq h}$ is prime. Its irreducible variety coincides with $\hat{\mathcal G}^{\leq h}(\polyd )$ and its dimension equals to $\Lambda_{d,l} $ .
\end{thm}

\begin{proof}
The affine variety $\hat{\mathcal G}^{\leq h}(\polyd )$ is irreducible because it is the image of the linear space $\mathcal L^{\leq h}(\polyd ) $  under the polynomial map $ \exp_{\otimes_h}\circ\Psi_H^*$. This map has a polynomial inverse, namely $\log_{\otimes_h}\circ\Phi_H^* $. Hence the dimension of $\hat{\mathcal G}^{\leq h}(\polyd )$ agrees with that of $\mathcal L^{\leq h}(\polyd )$, which is given from Lemma \ref{dim_lie}.
\end{proof}

We conclude the section by defining an explicit way to compute the coordinates of each element in the image of $\log_{\otimes_h}$. These are the \textbf{shuffle eulerian map} and \textbf{quasi-shuffle eulerian map}, the maps $e_1^{\shuffle}$, $e_1^{\qshuffle}: T(\polyd) \to T(\polyd)$ defined as
\begin{align*}
e_1^{\shuffle} = \sum_{n\geq 1}\frac{(-1)^{n-1}}{n} \shuffle^{\circ (n-1)} \circ \tilde{\delta}^{\circ (n-1)} \, , \quad
e_1^{\qshuffle} = \sum_{n\geq 1}\frac{(-1)^{n-1}}{n} \qshuffle^{\circ (n-1)} \circ \tilde{\delta}^{\circ (n-1)}\, ,
\end{align*}
where we use the convention that $\qshuffle^{\circ (0)} \circ \tilde{\delta}^{\circ (0)}=\shuffle^{\circ (0)} \circ \tilde{\delta}^{\circ (0)}$ is the projection from $T(\polyd) $ to $\varepsilon \KK$ and  the symbols $\qshuffle^{\circ (k)}$, $\shuffle^{\circ (k)}$ $\tilde{\delta}^{\circ (k)}$ describe  iterated products and reduced coproducts. For instance, one has
\begin{align*}
e_1^{\shuffle}(\uw{\varepsilon}) &= e_1^{\qshuffle}(\uw{\varepsilon}) = 0\,, \quad \quad 
e_1^{\shuffle}([i]) = e_1^{\qshuffle}([i]) = [i] \, ,\\
e_1^{\shuffle}([i][j]) &= \frac{1}{2}([i][j]-[j][i])\,, \quad e_1^{\qshuffle}([i] [j]) = \frac{1}{2}([i][j] -[j] [i]) + \frac{1}{2}[ij]\,.
\end{align*}
These two maps allow us to compute the coordinates of $\log_{\otimes_h}$ on group-like elements.

\begin{prop}\label{lm:adjoint_log}
Let $\vv\in \mathcal G^{\leq h}(\polyd)$ and $\vw\in \hat{\mathcal G}^{\leq h}(\polyd)$. 
For any word $w$ with $||w||\leq h$ one has:
\begin{equation}\label{log_eul}
\begin{split}
\langle \log_{\otimes_h} (\vv), w\rangle=\langle \vv, e_1^{\shuffle}(w)\rangle \, ,\quad 
\langle \log_{\otimes_h} (\vw), w\rangle=&\langle \vw, e_1^{\shuffle}(w)\rangle \, .
\end{split}
\end{equation}
\end{prop}

\begin{proof}
To prove the result, we use the following fact, called the duality between product and coproduct associated to $\otimes $ and $\delta$
$$\langle (\vv - \varepsilon)^{\otimes_h n}, \uw{w} \rangle = \langle \vv^{\otimes n}, \tilde{\delta}^{\circ ( n-1)} \uw{w} \rangle_{T(\polyd)^{\otimes n}} \, ,$$
This identity is a standard result in the Free Lie algebra literature, see e.g. \cite[Prop. 1.9]{reutenauer1993free}. Using the duality above, and applying \cref{shuffle_Lie} we have:
\begin{align*}
\langle\log_{\otimes_h}(\vv) , \uw{w} \rangle &= 
\sum_{n\geq 1} \frac{(-1)^{n-1}}{n} \langle (\vv - \uw{\varepsilon})^{\otimes_h n}, \uw{w} \rangle =
\sum_{n\geq 1} \frac{(-1)^{n-1}}{n} \langle \vv^{\otimes n}, \tilde{\delta}^{\circ (n-1)} \uw{w} \rangle_{T(\polyd)^{\otimes n}}\\
&=\sum_{n\geq 1} \frac{(-1)^{n-1}}{n} \langle \vv, \shuffle^{\circ (n-1)}\tilde{\delta}^{\circ (n-1)} \uw{w} \rangle \\
&=\langle \vv, \sum_{n\geq 1} \frac{(-1)^{n-1}}{n} \shuffle^{\circ (n-1)}\tilde{\delta}^{\circ (n-1)} \uw{w} \rangle = \langle \vv, e_1^{\shuffle} (\uw{w})\rangle \, ,
\end{align*}
which concludes one equality.
The remaining follows via the same computations. 
\end{proof}

The map $e_1^{\shuffle}$ will be crucial in writing down the reachability equations \eqref{eq_reach_eq}.

\section{Universal varieties and Chen-Chow Theorem for discrete signatures}\label{sec:conj}
We pass now to introduce the natural varieties associated with each  homogeneous component of the variety 
$\hat{\mathcal G}^{\leq h}(\polyd )$, by projecting on a specific weight. In what follows,  we denote by  $\pi^h: T^{\leq h}(\polyd) \to T^h(\polyd)$ the canonical projection on this finite-dimensional vector space.  We denote by $\Dsign_{w} $ with $w\in \mathcal{W}(\MS_d)$ $||w||= h$  the coordinates of $T^{\leq h}(\polyd)$ and we set by $\KK[\Dsign^{(h)}]$ the ring of polynomial function on $ T^h(\polyd)$ over a generic field $\KK$ of characteristic zero.
\begin{defin}\label{def:quasi-shuffle-variety}
Fix $d, h$ integers $\geq 1$. We define the \textbf{quasi-shuffle variety}
 $\hat{\mathcal{G}}_{d,h}$ to be the zero set of the ideal $\hat{G}_{d,h}= \hat{G}_{d,\leq h}\cap\mathbb{C}[\Dsign^{(h)}]$  in the projective space $\mathbb{P}^{n_{d,h}-1}(\mathbb{C})$.
\end{defin}
This is trivially a homogeneous prime ideal in the polynomial ring over $T^h(\polyd)$. Moreover, its dimension can be explicitly computed.

\begin{thm}\label{thm:dimension}
For any $d, h$  the dimension of the variety $ \hat{\mathcal{G}}_{d, h}$ is $\Lambda_{d, h}-1$ with $\Lambda_{d, h}$ is given by \eqref{eq:Lambda}.
\end{thm}

\begin{proof}
We follow the strategy laid out in \cite[Theorem 6.1]{amendola2019varieties}.
Specifically, we show that at the level of affine varieties $ \pi^h$ is a generically $h$-to-$1$ map on $\hat{\mathcal G}^{\leq h}(\polyd)$, thereby obtaining that the dimension of the affine zero set of the ideal $  \hat{G}_{d, h}$ is given by the dimension of $\hat{\mathcal G}^{\leq h}(\polyd)$ and the result will follow from Theorem \ref{thm_dim_1}. To show that $\pi^h$ is generically $h$-to-$1$, let $\vv \in \hat{\mathcal{G}}_{d, h}$ such that $\langle \vv, [1]^{\qshuffle h}\rangle \neq 0$, and write
$$\vv = \sum_{\substack{w \in \mathcal W (\MS_d) \\ ||w || = h}} \alpha_{w} \uw{w}\, .$$
If $\vv = \pi^h(\vw)$ for $\vw \in \hat{\mathcal G}^{\leq h}(\polyd)$, then 
$$\langle \vw,  [1]\rangle^h = \langle \vw, [1]^{\qshuffle h}\rangle = \langle \vv, [1]^{\qshuffle h}\rangle \, ,$$
because all elements in $[1]^{\qshuffle h}$ have weight $h$.
This equation determines a non-zero value for $\langle \vw,  [1]\rangle$ up to an $h$-root of unity of $1$.

Now for any $w $ of weight $l < h$, note that
$$\langle \vv, \uw{ w} \qshuffle [1]^{\qshuffle h-l}\rangle  = \langle \vw,  \uw{w} \qshuffle [1]^{\qshuffle h-l}\rangle =
\langle \vw,  \uw{w} \rangle \langle \vw,[1]\rangle^{ h-l} \, .$$
This determines 
$$\langle \vw,  \uw{w }\rangle = \langle \vv,\uw{  w} \qshuffle [1]^{\qshuffle h-l}\rangle/_{\langle \vw,[1]\rangle^{ h-l}}$$
We conclude that for $\vv$ in an open set of $\hat{\mathcal{G}}_{d, h}$, there are $h$ many values of $\vw$ that map to $\vv$. 
\end{proof}

Coming back to the discrete signatures variety, and looking at the equivalent version of the quasi-shuffle varieties for signatures of paths, see \cite[Section 4]{amendola2019varieties}, it turns out that the family of quasi-shuffle varieties $\hat{\mathcal{G}}_{d, h}$ should coincide with the family  $\mathcal V_{d, h}$  of universal varieties but this equality seems much harder to prove. 
The big reason behind this difficulty is the lack of a Chen-Chow theorem in the context of discrete signatures, relating the image of $x\to \Dsign^{\leq h}(x)$ with $\hat{\mathcal G}^{\leq h}(\polyd )$.

\begin{conj}[Chen-Chow theorem for discrete signatures]\label{thm_chen_chow}
Let $\KK$ be an algebraically closed field.
For any $d$,$h$ integers $\geq 1$ and $ g\in \hat{\mathcal G}^{\leq h}(\polyd )$ there exists a time series  $x\in \KK^{d\times N} $ for some $N\geq 1$ such that $\Dsign^{\leq h}(x)= g $.
\end{conj}

As it was already pointed out in \cite[Section 3]{Tapia20}, in the case $\KK= \mathbb{R}$ Conjecture \ref{thm_chen_chow} does not hold because of trivial counterexamples.  
However, if we can establish Conjecture \ref{thm_chen_chow} in the complex field, we would obtain $\V_{d, h}=\hat{\mathcal{G}}_{d, h}$.

\begin{prop}\label{conj:dim}
For any  $d, h$ integers $\geq 1$ one has $\V_{d, h}\subset \hat{\mathcal{G}}_{d, h}$. Moreover, if Conjecture \ref{thm_chen_chow} holds true when $\KK= \mathbb{C}$ then  $\V_{d, h}=\hat{\mathcal{G}}_{d, h}$.
\end{prop}
\begin{proof}
To prove the first inclusion it is sufficient to show that $\V_{d, h, N}\subset \hat{\mathcal{G}}_{d, h}$
for any integer $N\geq 1$. However, by projecting at weight $h$ the quasi-shuffle identity  in Theorem \ref{thm_prop_disc_sig} we immediately see that for any integer $N\geq 1$ one has 
\[\text{Im}\Dsign^h\subset \pi^h( \hat{\mathcal G}^{\leq h}(\polyd ))\]
from which we obtain the first inclusion. From Theorem \ref{thm_chen_chow} together with  Definition \ref{def_quasi-shuffle} we deduce immediately that there exists  an integer $\bar{N}\geq 1$ such that the ideal of the variety $\V_{d, h, \bar{N}}^{\text{af}}$ denoted by $I_{d, h, \bar{N}}$ satisfies 
\[I_{d, h, \bar{N}} \subset \hat{G}_{d,h}\,.\]
By passing to projective varieties, we obtain $\hat{\mathcal{G}}_{d,h}\subset \V_{d, h, \bar{N}}\subset \mathcal{V}_{d,h}$ and we conclude.
\end{proof}

In what follows, we illustrate a strategy to prove Conjecture \ref{thm_chen_chow}, through which we will prove the case $h=2$. 
With the results in section \ref{sec_Lie} we show that \cref{thm_chen_chow} is equivalent to a simpler property on the finite-dimensional vector space $\polyd^{\leq h} $, that we call \textbf{reachability}.

%In the rest of this section we present some important consequences of this conjecture. Specifically, we present some dimension results and an enumeration result.

%The following definition and lemma outline a strategy for showing \cref{conj:dim}.
\begin{defin}\label{defin:reachability}
A given  vector $\vv \in \polyd^{\leq h}$ is said to be \textbf{reachable} if  there exists a time series $x \in \KK^{d \times N}$ for some $N>1$ such that 
\begin{equation}\label{reac_eq_1}
\log_{\otimes_h}\circ\, \Phi^{*}_H \Dsign^{\leq h}(x) = \vv\,.
\end{equation}
\end{defin}

 Even if the map $\log_{\otimes_h}\circ \,\Phi^{*}_H$ in \eqref{reac_eq_1} seems arbitrary, it is natural to use it on $\Dsign^{\leq h}(x)$ because one sees immediately  from  the diagram  \eqref{cd:level_h_lie} that it provides a global chart to project the group $\hat{\mathcal G}^{\leq h}(\polyd)$ onto the vector space $\mathcal L^{\leq h}(\polyd)$, which we described explicitly in section \ref{sec_Lie}. By projecting the equation over a basis of  $\mathcal L^{\leq h}(\polyd)$ we obtain an explicit set of polynomial equation to describe this definition.

\begin{lm}\label{lm:reachable}
Fix $d, h$ integers $\geq 1$. an element $\vv\in \polyd^{\leq h}$ is reachable if and only if there exists a time series $x \in \KK^{d \times N}$ for some $N>1$ that satisfies the following system of equations 
\begin{equation}\label{eq_reach_eq}
\begin{cases}
\langle \Dsign (x), \uw{I}\rangle = \langle \vv, \uw{I}\rangle \\
\langle \Dsign (x), \Phi_H\circ  e_1^{\shuffle} \, (b(J))\rangle = 0
\end{cases}
\end{equation}
for all  $I \in \MS_d$ of weight at most  $h$ and all Lyndon words  $J $ built from $\MS_d$ of weight at most  $h$ such that $|J|\geq 2$ with $b(J)$ the Lyndon bracket associated to it. We call the equations in \eqref{eq_reach_eq} the \textbf{reachability equations}. 
\end{lm}
\begin{proof}
Looking at the diagram  \ref{cd:level_h_lie}, we observe that $\log_{\otimes_h}\circ \,\Phi^{*}_H \Dsign^{\leq h}(x) = \vv$ if and only if $\langle\log_{\otimes_h}\circ \,\Phi^{*}_H\Dsign(x), \bv\rangle =\langle \vv, \bv \rangle $ for all $\bv$ running over a basis of $\mathcal L^{\leq h}(\polyd)$. Thanks to Lemma \ref{dim_lie} one has
\[ \langle \log_{\otimes_h} \circ \,\Phi^{*}_H \Dsign(x), b(K)\rangle =\langle \vv, b(K) \rangle\]  
for all Lyndon words  $K $ of weight at most  $h$ starting from the alphabet $\MS_d$. By using \cref{lm:adjoint_log} the left-hand side in the above equation becomes
\begin{align*}
\langle \log_{\otimes_h}\circ \,\Phi^{*}_H \Dsign(x),b(K) \rangle =&
\langle  \Phi^{*}_H\Dsign(x),e_1^{\shuffle} (b(K)) \rangle\\
 =&
\langle  \Dsign(x),\Phi_H\circ  e_1^{\shuffle}(b(K)) \rangle\, .
\end{align*}
We obtain the desired equations once we note that whenever $|K | = 1$, $K$ is simply an  element of $\MS_d$ of weight at most  $h$, with $b(K)=K$ and  $\Phi_H\circ e_1^{\shuffle}(K)  =K$. Furthermore, the map $\Phi_H \circ e_1^{\shuffle}$ is graded with respect to the  word's length, so $\langle  \vv,\Phi_H\circ  e_1^{\shuffle}(b(K)) \rangle = 0 $ for $|K| > 1$.
\end{proof}
\begin{rem}
The choice of the basis given by Lyndon words to derive the reachability equations is completely arbitrary. Other possible basis for $\mathcal L^{\leq h}(\polyd)$ are studied in \cite[Chapter 4]{reutenauer1993free}, and some may lead to more intuitive structure.
%General set of equations We call the equations in \cref{defin:reachability} the \textbf{reachability} equations. These are split into \textbf{levels} according to the length of $w$. In this way, the reachability equations of level $k$ correspond to all equations where $w$ words have length $k$. Our strategy is to show that all $\vv $ of degree at most $h$ are reachable: the following lemma establishes that this is enough to infer \cref{conj:dim}.
\end{rem}

We now link the notion of reachability with \cref{thm_chen_chow}.

\begin{prop}\label{lm:strategy}
Fix $d, h$ integers $\geq 1$.  \cref{thm_chen_chow} holds with parameters $h$ and $d$ if and only if every element of $\vv \in \polyd^{\leq h}$ is reachable.
%Consider $\Dsign^{\leq h} $ as a map $\Dsign^{\leq h} :(\KK^d)^m \to \hat{\mathcal G}^{\leq h}(\polyd)$.Then, for some integer $m$, we have $\im \Dsign^{\leq h} = \hat{\mathcal G}^{\leq h}(\polyd)$.
\end{prop}

%If $a \in (\KK^d)^{N}, b\in (\KK^d)^{M}$ are two time series, we denote its concatenation by $a|b$.Specifically, it denotes the time s-series in $\KK^d$ with $M+N$ vectors resulting from appending $b$ to $a$. For the proof of this lemma, we use the following fact without proof:

\begin{proof}
Supposing \cref{thm_chen_chow} true we can fix $ \vv\in \polyd^{\leq h}$
and apply  \cref{thm_chen_chow} to the element $\exp_{\otimes_h}\circ \Psi_H^*(\vv)$, which by construction belongs to $\hat{\mathcal G}^{\leq h}(\polyd )$. By inverting the maps $\exp_{\otimes_h}$ and $\Psi_H^*$ we obtain \eqref{reac_eq_1}.

On the other hand, supposing that every element of $ \vv \in \polyd^{\leq h}$ is reachable we derive  \cref{thm_chen_chow}. We fix $g \in \hat{\mathcal G}^{\leq h}(\polyd)$ and consider $ \Phi_H^*(g)\in \mathcal G^{\leq h}(\polyd)$. By applying \cref{thm:chow}  we can find $\vv^1, \dots , \vv^m \in \polyd$ such that
$$\Phi_H^*(g) = \exp_{\otimes_h}(\vv^1) \otimes_h \cdots \otimes_h\exp_{\otimes_h}(\vv^m)\, . $$
Since every element $\vv^j $,  $j=1\,, \ldots\,, m$ is reachable there exists a time series $x(\vv^j)$  such that $\log_{\otimes_h}\circ \Phi^{*}_H \Dsign^{\leq h}(x(\vv^j)) = \vv^j$. Plugging this relation in the equation above, we obtain
$$
\Phi^*_H(g) = \Phi^*_H\Dsign^{\leq h} (x(\vv^1)) \otimes_h\cdots  \otimes_h \Phi^*_H\Dsign^{\leq h} (x(\vv^m))\,. $$
Using the homomorphism property of $\Phi^*_H$ and \eqref{trunc_chen} one has 
$$
\Phi^*_H(g) = \Phi^*_H(\Dsign^{\leq h} (x(\vv^1)|\cdots  | (x(\vv^m)))\,. $$
Because $\Phi^{*}_H$ is an isomorphism, we conclude.
\end{proof}
We display the equations necessary to solve the reachability problem for weight two. Writing every element of $\vv\in \polyd^{\leq 2}$ as a sum
\[\vv= \sum_{i=1}^d \vv^{i}[i] + \sum_{1\leq i\leq j\leq d}\vv^{ij} [ij]\,,\]
and using the properties of the element $e_1^{\shuffle}$, the equations \eqref{eq_reach_eq} become
\begin{equation}\label{level_2eq}
\begin{cases}
\langle \Dsign (x), [i]\rangle = \vv^{i} & 1\leq i\leq d  \\
\langle \Dsign (x),[ij]\rangle = \vv^{ij} & \;\;1\leq i\leq j \leq d
\\
\langle \Dsign (x), [i][j]-[j] [i]\rangle = 0 &\;\;1\leq i< j \leq d
\end{cases}
\end{equation}
%Consider the timeseries $x = x(\vv^1) | \cdots | x(\vv^m) $ in.
%Then \eqref{eq_chen_equation}, together with the fact that $\Phi^*_H$ is an algebra homomorphism (see \cref{eq:phistar}) gives us:
%\[\Phi^*_H\Dsign^{\leq h} (x) = \Phi^*_H\Dsign^{\leq h} (x(\vv^1)) \otimes_h\cdots  \otimes_h \Phi^*_H\Dsign^{\leq h} (x(\vv^m)) = \vw\,.\]
%This shows that the map $\Phi^{*}_H \circ \Dsign^{\leq h}$ is surjective.

\begin{thm}\label{lm:h2constructionX}
Assume $\KK= \mathbb{C}$. Then every element of $\vv \in \polyd^{\leq h}$ is reachable when $h=2$ and $d\geq 1$. 
\end{thm}

\begin{proof}
%By invariance of translation it is sufficient to solve the equations
%\[
%\begin{cases}
%\langle \Dsign (x), [i]\rangle = 0 & [i]=\mathtt{1}\,, \ldots\,, \mathtt{d} \\
%\langle \Dsign (x),\mathtt{ij}\rangle = \vv^{\mathtt{ij}} & [i], [j]\in \{1\,, \ldots\, d\}\,,\; \;[i]\leq [j] \\
%\langle \Dsign (x), [i]\otimes [j]-[j]\otimes [i]\rangle = 0 & [i], [j]\in \{1\,, \ldots\, d\}\,,\;\;[i]< [j]
%\end{cases}
%\]
First, we construct a time series $x = x(\vv)$ that satisfies the first two lines of \eqref{level_2eq}. 
From this solution, we construct another one.
This will also satisfy the last equation, which is amenable to algebraic manipulations on level one.

Let us find the time series $x(\vv)$ solving the first two lines of \eqref{level_2eq} by a dimension argument.  Let $\U$ be a basis of $\polyd^{\leq 2}$. 
Let $Y$ be the variety in $\KK^{|\U|}$ given by the image of the following map from $\KK^d$
$$ z \mapsto (p(z))_{p \in \U} \, . $$
This is an irreducible variety, as observed in the proof of \cref{thm:chain}.

We note that $Y$ is a variety that is not contained in any hyperplane.
Indeed, assume for sake of contradiction that $Y \subseteq \{ \sum_{p \in \mathcal U} \alpha_p z_p \}$, this can be written as the polynomial $\sum_{p \in \mathcal U} \alpha_p p(z)$ is identically zero, which is impossible because $\U$ is a basis.

Say it has dimension $e$. Fix $N = \lceil |\U|/{e} \rceil$. The linear relations $\langle \Dsign (x), [i]\rangle = x_N^i- x_1^i=\vv^{i}$ for $i\in \{1\,, \ldots\,, d \}  $ trace out an affine space $\LL$ of codimension $|\U|$ in $\KK^N$.

Given that the $N$-fold Cartesian product $Y^N$ has dimension $eN \geq |\U|$, and since $Y^N$ is not contained in any hyperplane, the intersection $\LL \cap Y^N$ is non-empty due to the fundamental theorem of analysis.
The time series $x$ in this intersection space is the desired $x(\vv)$.

To establish the equations on level two, for a time  series $x$ let $
\overleftarrow{x}$ be the reversed-time  time series of $x$.
Then, one verifies that for any time series $x$, the time series $ g(x) = x |\overleftarrow{x} $  satisfies the equations of weight two, while having an amenable level one result, that is:
\[
\begin{cases}
\langle \Dsign (g(x)), [i]\rangle = 2\vv^{i} & \; \; 1\leq i\leq d \\
\langle \Dsign (g(x)),[ij]\rangle =2^2 \vv^{ij} & \; \;1\leq i\leq j\leq d \\
\langle \Dsign (g(x)), [i][j]-[j] [i]\rangle = 0 & 1\leq i< j\leq d 
\end{cases}
\]
It follows that $g( 2^{-1}x(\vv) )$ satisfies the desired equations, and the theorem is proven.
\end{proof}
%\subsection{Length three equations\label{sec:len_three_eqs}}
\begin{rem}
We note that the dimension argument presented in the proof above works for any fixed weight, and displays the intuition that given a large enough time series size, a simple dimension argument allows us to construct a solution.  One can conjecture that a similar dimension-bound argument is enough to find a solution to the equations \eqref{level_2eq}.
One can conjecture that a similar dimension-bound argument is enough to find a solution to the equations \eqref{level_2eq}.
We were not able to find one such argument, due to the symmetric constraint $\langle \Dsign (x), [i][j]-[j] [i]\rangle=0$. When $h>2$ the complexity of the reachability equation \eqref{eq_reach_eq} becomes even more challenging without even a clear expression in any order. For example in case $h=3$ writing any generic element $\vv\in \polyd^{\leq 3}$ as the sum
\[\vv= \sum_{i=1}^d \vv^{i}[i] + \sum_{1\leq i\leq j\leq d}\vv^{ij} [ij]+ \sum_{1\leq i\leq j\leq k\leq d}\vv^{ijk}[ijk]\,,\]
then the equations  \eqref{eq_reach_eq} become
\begin{equation}\label{eq_reach_eq3}
\begin{cases}
\langle \Dsign (x), [i]\rangle = \vv^{i} & \;\;1\leq i\leq d  \\
\langle \Dsign (x),[ij]\rangle = \vv^{ij} & \;\;1\leq i\leq j\leq d\\
\langle \Dsign (x),[ijk]\rangle = \vv^{ijk} & \;\;1\leq i\leq j\leq k\leq d\\
\langle \Dsign (x), [i] [j]-[j] [i]\rangle = 0 & \;\;1\leq i< j\leq d\\
\langle \Dsign (x), [i] [jk]-[jk] [i]\rangle = 0 & \;\;1\leq i\leq j\leq k\leq d\\

\langle \Dsign (x),  \Phi_H\circ  e_1^{\shuffle}(\lb\; [i],\lb \;[i],[j]\; \rb\rb )\rangle = 0 & \;\;1\leq i< j\leq d\\
\langle \Dsign (x),  \Phi_H\circ  e_1^{\shuffle}(\lb \lb \;[i], [j]\;\rb,  [k]\;\rb)\rangle = 0 & \;\;1\leq i< j<k\leq d

\end{cases}
\end{equation}
Thanks to the special  identities 
\[e_1^{\shuffle}(\lb\; [i],\lb \;[i],[j]\; \rb\rb )= \lb [i],\lb \;[i],[j]\; \rb\rb\]%=[i][i][j]+ [j][i][i]- 2[i][j][i]
\[e_1^{\shuffle}(\lb \lb \;[i], [j]\;\rb,  [k]\;\rb)=\lb \lb \;[i], [j]\;\rb,  [k]\;\rb\]%=[i][j][k]+[k][j][i]- [k][i][j]- [j][i][k]
 system \eqref{eq_reach_eq3} becomes
\[
\begin{cases}
\langle \Dsign (x), [i]\rangle = \vv^{i} \\
\langle \Dsign (x),[ij]\rangle = \vv^{ij} \\
\langle \Dsign (x),[ijk]\rangle = \vv^{ijk} \\
\langle \Dsign (x), [i] [j]-[j] [i]\rangle = 0 \\
\langle \Dsign (x), [i] [jk]-[jk] [i]\rangle = 0 \\

\langle \Dsign (x),  \lb\; [i],\lb \;[i],[j]\;\rb\rb+ \frac{1}{2}([j][ii]+ [ii][j]- [ij][i]-[i][ij])\rangle = 0\\
\langle \Dsign (x),\lb \lb \;[i], [j]\;\rb,  [k]\;\rb + \frac{1}{2}([ij][k]+ [k][ij]-[j][ki]-[ki][j])\rangle = 0 

\end{cases}
\]
For these equations we are not able to provide the existence of a solution.
It is possible to show that $e_1^{\shuffle}(b(K))\neq K$ for the Lyndon word $K = 1112$ of weight $4$, therefore from height four and above, apart from explicit computations we lack a general combinatorial formula for the reachability equations \eqref{eq_reach_eq} in coordinates.
\end{rem}

\subsection*{Further Work}

This is an exploratory pilot work on some algebraic geometric structures built from invariant quantities on a time series. Therefore, some combinatorial and algebraic questions are left to be resolved in this paper that we wish to tackle in future work, or wish to see resolved by the community. Here is a subset of these questions:

\begin{itemize}
\item  What is the role of the length  between the parameters $N$ and $h$  in $\mathcal{V}_{d,h, N}$ when we consider the identifiability of a time series?  Can we expect results similar to \cite[Section 6]{amendola2019varieties} or \cite{pfeffer2019learning}?
\item What happens in a field of positive characteristic? Can we provide a characterisation of quasi-shuffles in positive characteristic, very much in the same way as Hazewinkel in \cite{hazewinkel2001algebra}?

\item Supposing that now the underlying time series $x$ is random. By construction, $\Dsign^h(x)$ is a well-defined random variable in $\V_{d, h}$. What can we say about the varieties obtained when we compute the expectation of such quantities?
\end{itemize}

\subsection*{Aknowledgments}
CB gratefully acknowledges funding support from the ERC Starting Grant Low Regularity Dynamics via Decorated Trees (LoRDeT). The views and opinions expressed are those of the author(s) only and do not necessarily reflect those of the European Union or the European Research Council. Neither the European Union nor the granting authority can be held responsible for them. He also acknowledges the support of TU Berlin, where he was employed when this project was initiated. RP was supported by the Max Planck Institute for Mathematics in the Sciences during part of this project. Both authors would like to thank Sylvie Paycha and Bernd Sturmfels for fruitful conversations. We are also grateful to Ángel Ríos Ortiz, Pierpaola Santarsiero, and Nikolas Tapia for their suggestions and comments on discrete signatures, tensor algebras, and algebraic varieties. The authors thank the reviewers for their helpful comments, in particular for drawing attention to \cref{rem:power_degree}.

\bibliographystyle{alpha}
\bibliography{bibli}

\end{document}